\documentclass[11pt]{article}
\usepackage{amssymb}
\usepackage{amsfonts}
\usepackage{amsmath}

\newtheorem{theorem}{Theorem}[section]
\newtheorem{acknowledgement}[theorem]{Acknowledgement}
\newtheorem{definition}[theorem]{Definition}

\newtheorem{lemma}[theorem]{Lemma}

\newtheorem{proposition}[theorem]{Proposition}
\newtheorem{remark}[theorem]{Remark}

\newenvironment{proof}[1][Proof]{\textbf{#1.} }{\hfill\rule{0.5em}{0.5em}}
{\catcode`\@=11\global\let\AddToReset=\@addtoreset
\AddToReset{equation}{section}

\AddToReset{theorem}{section}

\begin{document}

\title{Remarks on some quasilinear equations with gradient terms and measure
data}
\author{Marie-Fran\c{c}oise BIDAUT-VERON\thanks{%
Laboratoire de Math\'{e}matiques et Physique Th\'{e}orique, CNRS UMR 7350,
Facult\'{e} des Sciences, 37200 Tours France. E-mail: veronmf@univ-tours.fr}
\and Marta GARCIA-HUIDOBRO\thanks{%
Departamento de Matematicas, Pontifica Universidad Catolica de Chile,
Casilla 307, Correo 2, Santiago de Chile. E-mail: mgarcia@mat.puc.cl} \and %
Laurent VERON\thanks{%
Laboratoire de Math\'{e}matiques et Physique Th\'{e}orique, CNRS UMR 7350,
Facult\'{e} des Sciences, 37200 Tours France. E-mail: veronl@univ-tours.fr}}
\date{.}
\maketitle

\begin{abstract}
Let $\Omega \subset \mathbb{R}^{N}$ be a smooth bounded domain, $H$ a
Caratheodory function defined in $\Omega \times \mathbb{R\times R}^{N},$ and
$\mu $ a bounded Radon measure in $\Omega .$ We study the problem%
\begin{equation*}
-\Delta _{p}u+H(x,u,\nabla u)=\mu \quad \text{in }\Omega ,\qquad u=0\quad
\text{on }\partial \Omega ,
\end{equation*}
where $\Delta _{p}$ is the $p$-Laplacian ($p>1$)$,$ and we emphasize the
case $H(x,u,\nabla u)=\pm \left\vert \nabla u\right\vert ^{q}$ ($q>0$). We
obtain an existence result under subcritical growth assumptions on $H,$ we
give necessary conditions of existence in terms of capacity properties, and
we prove removability results of eventual singularities. In the
supercritical case, when $\mu \geqq 0$ and $H$ is an absorption term, i.e. $%
H\geqq 0,$ we give two sufficient conditions for existence of a nonnegative
solution.\bigskip \bigskip \bigskip \bigskip \bigskip \bigskip
\end{abstract}

\pagebreak \medskip

\section{Introduction\protect\bigskip}

Let $\Omega $ be a smooth bounded domain in $\mathbb{R}^{N}(N\geqq 2)$. In
this article we consider problems of the form%
\begin{equation}
-\Delta _{p}u+H(x,u,\nabla u)=\mu \hspace{0.5cm}\text{in }\Omega ,
\label{PG}
\end{equation}%
where $\Delta _{p}u= {div}(\left\vert \nabla u\right\vert ^{p-2}\nabla
u) $ is the $p$-Laplace operator, with $1<p\leqq N,$ $H$ is a Caratheodory
function defined in $\Omega \times \mathbb{R\times R}^{N}$, and $\mu $ is a
possibly signed Radon measure on $\Omega .$ We study the existence of
solutions for the Dirichlet problem in $\Omega $
\begin{equation}
-\Delta _{p}u+H(x,u,\nabla u)=\mu \hspace{0.5cm}\text{in }\Omega ,\qquad u=0%
\hspace{0.5cm}\text{on }\partial \Omega ,  \label{P0}
\end{equation}%
and some questions of removability of the singularities. Our main motivation
is the case where $\mu $ is nonnegative, $H$ involves only $\nabla u$, and
either $H$ is nonnegative, hence $H$ is an absorption term, or $H$ is
nonpositive, hence $H$ is a source one. The model cases are
\begin{equation}
-\Delta _{p}u+\left\vert \nabla u\right\vert ^{q}=\mu \hspace{0.5cm}\text{in
}\Omega ,  \label{abs}
\end{equation}%
where $q>0,$ for the absorption case and
\begin{equation}
-\Delta _{p}u=\left\vert \nabla u\right\vert ^{q}+\mu \hspace{0.5cm}\text{in
}\Omega .  \label{sou}
\end{equation}%
for the source case. \medskip

The equations without gradient terms,
\begin{equation}
-\Delta _{p}u+H(x,u)=\mu \hspace{0.5cm}\text{in }\Omega ,  \label{pbu}
\end{equation}%
such as the quasilinear Emden-Fowler equations
\begin{equation*}
-\Delta _{p}u\pm \left\vert u\right\vert ^{Q-1}u=\mu \hspace{0.5cm}\text{in }%
\Omega ,
\end{equation*}%
where $Q>0,$ have been the object of a huge literature when $p=2.$ In the
general case $p>1,$ among many works we refer to \cite{Be}, \cite{Bi1}, \cite%
{Bi2} and the references therein, and to \cite{BiHuVe} for new recent
results in the case of absorption.

We set
\begin{equation}
Q_{c}=\frac{N(p-1)}{N-p},\qquad q_{c}=\frac{N(p-1)}{N-1},\qquad
(Q_{c}=\infty \text{ if }p=N),\qquad \tilde{q}=p-1+\frac{p}{N}  \label{defi}
\end{equation}%
(hence $q_{c},\tilde{q}<p<N$ or $q_{c}=\tilde{q}=p=N$), and
\begin{equation}
q_{\ast }=\frac{q}{q+1-p},  \label{qch}
\end{equation}%
(thus $q_{\ast }=q^{\prime }$ in case $p=2).\medskip $

In Section 2 we recall the main notions of solutions of the problem $-\Delta
_{p}u=\mu $, such as weak solutions, renormalized or locally renormalized
solutions, and convergence results. In Section 3 we prove a general
existence result for problem (\ref{P0}) in the subcritical case, see Theorem %
\ref{subc}. Then in Section 4 we give necessary conditions for existence and
removability results for the local solutions of problem (\ref{PG}),
extending former results of \cite{HaMaVe} and \cite{Ph}, see Theorem \ref%
{remo}. In Section 5 we study the problem (\ref{P0}) in the supercritical
case, where many questions are still open. We give two partial results of
existence in Theorems \ref{one} and \ref{two}. Finally in Section 5 we make
some remarks of regularity for the problem
\begin{equation*}
-\Delta _{p}u+H(x,u,\nabla u)=0\hspace{0.5cm}\text{in }\Omega .
\end{equation*}

\section{Notions of solutions \label{Reg}}

Let $\omega $ be any domain of $\mathbb{R}^{N}.$ For any $r>1,$ the capacity
$cap_{1,r}$ associated to $W_{0}^{1,r}(\omega )$ is defined by \medskip
\begin{equation*}
cap_{1,r}(K,\omega )=\inf \left\{ \left\Vert \psi \right\Vert
_{W_{0}^{1,r}(\omega )}^{r}:\psi \in \mathcal{D}(\omega ),\chi _{K}\leq \psi
\leq 1\right\} ,
\end{equation*}%
for any compact set $K\subset \omega ,$ and then the notion is extended to
any Borel set in $\omega .$ In $\mathbb{R}^{N}$ we denote by $G_{1}$ the
Bessel kernel of order 1 (defined by $\widehat{G_{1}}(y)=(1+\left\vert
y\right\vert ^{2})^{-1/2}),$ and we consider the Bessel capacity defined for
any compact $K\subset \mathbb{R}^{N}$ by
\begin{equation*}
Cap_{1,r}(K,\mathbb{R}^{N})=\inf \left\{ \left\Vert f\right\Vert _{L^{r}(%
\mathbb{R}^{N})}^{r}:f\geqq 0,G_{1}\ast f\geqq \chi _{K}\right\} .
\end{equation*}%
On $\mathbb{R}^{N}$ the two capacities are equivalent, see \cite{AdPo}.
\medskip

We denote by $\mathcal{M}(\omega )$ the set of Radon measures in $\omega ,$
and $\mathcal{M}_{b}(\omega )$ the subset of bounded measures, and define $%
\mathcal{M}^{+}(\omega ),$ $\mathcal{M}_{b}^{+}(\omega )$ the corresponding
cones of nonnegative measures. Any measure $\mu \in \mathcal{M}(\omega )$
admits a positive and a negative parts, denoted by $\mu ^{+}$ and $\mu ^{-}$%
. For any Borel set $E$, $\mu \llcorner E$ is the restriction of $\mu $ to $%
E;$ we say that $\mu $ is concentrated on $E$ if $\mu =\mu \llcorner E.$%
\medskip

For any $r>1,$ we call $\mathcal{M}^{r}(\omega )$ the set of measures $\mu
\in \mathcal{M}(\omega )$ which do not charge the sets of null capacity,
that means $\mu (E)=0$ for every Borel set $E\subset \omega $ with $%
cap_{1,r}(E,\omega )=0$. Any measure concentrated on a set $E$ with $%
cap_{1,r}(E,\omega )=0$ is called $r$-singular. Similarly we define the
subsets $\mathcal{M}_{b}^{r}(\omega )$ and $\mathcal{M}_{b}^{r+}(\omega )$%
.\medskip

For fixed $r>1,$ any measure $\mu \in \mathcal{M}(\omega )$ admits a unique
decomposition of the form $\mu =\mu _{0}+\mu _{s},$ where $\mu _{0}\in
\mathcal{M}^{r}(\omega )$, and $\mu _{s}=\mu _{s}^{+}-\mu _{s}^{-}$ is $r$%
-singular. If $\mu \geqq 0,$ then $\mu _{0}\geqq 0$ and $\mu _{s}\geqq 0$.

\begin{remark}
\label{imp}Any measure $\mu \in \mathcal{M}_{b}(\omega )$ belongs to $%
\mathcal{M}^{r}(\omega )$ if and only if there exist $f\in L^{1}(\omega )$
and $g\in (L^{r^{\prime }}(\omega ))^{N}$ such that $\mu =f+ {div}g,$
see \cite[Theorem 2.1]{BoGaOr}. However this decomposition is not unique; if
$\mu $ is nonnegative there exists a decomposition such that $f$ is
nonnegative, but one cannot ensure that $ {div}g$ is nonnegative.
\end{remark}

For any $k>0$ and $s\in \mathbb{R},$ we define the truncation $T_{k}(s)=\max
(-k,\min (k,s)).$ If $u$ is measurable and finite a.e. in $\omega $, and $%
T_{k}(u)$ belongs to $W_{0}^{1,p}(\omega )$ for every $k>0$, one can define
the gradient $\nabla u$ a.e. in $\omega $ by $\nabla T_{k}(u)=\nabla u.\chi
_{\left\{ \left\vert u\right\vert \leqq k\right\} }$ for any $k>0.$\medskip

For any $f\in \mathcal{M}^{+}\left( \mathbb{R}^{N}\right) ,$ we denote the
Bessel potential of $f$ by $J_{1}(f)=G_{1}\ast f.$

\subsection{Renormalized solutions}

Let $\mu \in \mathcal{M}_{b}(\Omega ).$ Let us recall some known results for
the problem%
\begin{equation}
-\Delta _{p}u=\mu \hspace{0.5cm}\text{in }\Omega ,\qquad u=0\quad \text{on }%
\partial \Omega ,  \label{mu}
\end{equation}

Under the assumption $p>2-1/N,$ from \cite{BoGa}, problem (\ref{mu}) admits
a solution $u\in W_{0}^{1,r}(\Omega )$ for every $r\in \left[ 1,q_{c}\right)
,$ satisfying the equation in $\mathcal{D}^{\prime }\left( \Omega \right) .$
When $p<2-1/N$, then $q_{c}<1;$ this leads to introduce the concept of
renormalized solutions developed in \cite{DMOP}, see also \cite{Mae}, \cite%
{TruWa}. Here we recall one of their definitions, among four equivalent ones
given in \cite{DMOP}.

\begin{definition}
\label{nor}Let $\mu =\mu ^{0}+\mu _{s}\in \mathcal{M}_{b}(\Omega )$, where $%
\mu ^{0}\in \mathcal{M}^{p}(\Omega )$ and $\mu _{s}=\mu _{s}^{+}-\mu
_{s}^{-} $ is $p$-singular. A function $u$ is a \textbf{renormalized
solution, }called\textbf{\ R-solution} of problem (\ref{mu}), if $u$ is
measurable and finite a.e. in $\Omega $, such that $T_{k}(u)$ belongs to $%
W_{0}^{1,p}(\Omega )$ for any $k>0,$ and $\left\vert \nabla u\right\vert
^{p-1}{\in }L^{\tau }(\Omega ),$ {for any }$\tau \in \left[ 1,N/(N-1)\right)
;$ and for any $h\in W^{1,\infty }(\mathbb{R})$ such that $h^{\prime }$ has
a compact support, and any $\varphi \in W^{1,s}(\Omega )$ for some $s>N,$
such that $h(u)\varphi \in W_{0}^{1,p}(\Omega ),$%
\begin{equation}
\int_{\Omega }\left\vert \nabla u\right\vert ^{p-2}\nabla u.\nabla
(h(u)\varphi )dx=\int_{\Omega }h(u)\varphi d\mu _{0}+h(\infty )\int_{\Omega
}\varphi d\mu _{s}^{+}-h(-\infty )\int_{\Omega }\varphi d\mu _{s}^{-}.
\label{norr}
\end{equation}%
\medskip
\end{definition}

As a consequence, any R-solution $u$ of problem (\ref{mu}) satisfies $%
\left\vert u\right\vert ^{p-1}\in L^{\sigma }(\Omega ),\forall \sigma \in %
\left[ 1,N/(N-p\right) .$ More precisely, $u$ and $\left\vert \nabla
u\right\vert $ belong to some Marcinkiewicz spaces
\begin{equation*}
L^{s,\infty }(\Omega )=\left\{ u\text{ measurable in }\Omega
:\sup_{k>0}k^{s}\left\vert \left\{ x\in \Omega :\left\vert u(x)\right\vert
>k\right\} \right\vert <\infty \right\} ,
\end{equation*}
see \cite{BoGa}, \cite{Be}, \cite{DMOP}, \cite{Kil}, and one gets useful
convergence properties, see \cite[Theorem 4.1 and \S 5]{DMOP} for the proof:

\begin{lemma}
\label{ldmop} (i) Let $\mu \in \mathcal{M}_{b}(\Omega )$ and $u$ be any
R-solution of problem (\ref{mu}). Then for any $k>0,$
\begin{equation*}
\frac{1}{k}\int_{\left\{ m\leqq u\leqq m+k\right\} }\left\vert \nabla
u\right\vert ^{p}dx\leq \left\vert \mu \right\vert (\Omega ),\forall m\geqq
0.
\end{equation*}

If $p<N$, then $u\in L^{Q_{c},\infty }(\Omega )$ and $\left\vert \nabla
u\right\vert \in L^{q_{c},\infty }(\Omega )$,
\begin{equation}
\left\vert \left\{ \left\vert u\right\vert \geqq k\right\} \right\vert \leqq
C(N,p)k^{-Q_{c}}(\left\vert \mu \right\vert (\Omega ))^{\frac{N}{N-p}%
},\qquad \left\vert \left\{ \left\vert \nabla u\right\vert \geqq k\right\}
\right\vert \leqq C(N,p)k^{-q_{c}}(\left\vert \mu \right\vert (\Omega ))^{^{%
\frac{N}{N-1}}}.  \label{sti}
\end{equation}

If $p=N$ (where $u$ is unique), then for any $r>1$ and $s\in \left(
1,N\right) ,$
\begin{equation}
\left\vert \left\{ \left\vert u\right\vert \geqq k\right\} \right\vert \leqq
C(N,p,r)k^{-r}(\left\vert \mu \right\vert (\Omega ))^{^{\frac{r}{p-1}%
}},\qquad \left\vert \left\{ \left\vert \nabla u\right\vert \geqq k\right\}
\right\vert \leqq C(N,p,s)k^{-N}(\left\vert \mu \right\vert (\Omega ))^{%
\frac{s}{N-1}}.  \label{sto}
\end{equation}%
(ii) Let $\left( \mu _{n}\right) $ be a sequence of measures $\mu _{n}\in
\mathcal{M}_{b}(\Omega ),$ uniformly bounded in $\mathcal{M}_{b}(\Omega ),$
and $u_{n}$ be any R-solution of
\begin{equation*}
-\Delta _{p}u_{n}=\mu _{n}\hspace{0.5cm}\text{in }\Omega ,\qquad
u_{n}=0\quad \text{on }\partial \Omega .
\end{equation*}%
Then there exists a subsequence $\left( \mu _{\nu }\right) $ such that $%
\left( u_{\nu }\right) $ converges $a.e.$ in $\Omega $ to a function $u,$
such that $T_{k}\left( u\right) \in W_{0}^{1,p}(\Omega ),$ and $(T_{k}\left(
u_{\nu }\right) )$ converges weakly in $W_{0}^{1,p}(\Omega )$ to $T_{k}(u),$
and $\left( \nabla u_{\nu }\right) $ converges $a.e.$ in $\Omega $ to $%
\nabla u.$
\end{lemma}

\begin{remark}
These properties do not require any regularity of $\Omega .$ If $\mathbb{R}%
^{N}\backslash \Omega $ is geometrically dense, i.e. there exists $c>0$ such
that $\left\vert B(x,r)\backslash \Omega \right\vert \geqq cr^{N}$ for any $%
x\in \mathbb{R}^{N}\backslash \Omega $ and $r>0,$ then (\ref{sto}) holds
with $s=N,$ and $C$ depends also on the geometry of $\Omega .$ Then $%
\left\vert \nabla u\right\vert \in L^{N,\infty }(\Omega ),$ hence $u\in
BMO(\Omega ),$ see \cite{DoHuMu}, \cite{Kil}. \medskip
\end{remark}

Next we recall the fundamental stability result of \cite[Theorem 3.1]{DMOP}:

\begin{definition}
For any measure $\mu =\mu ^{0}+\mu _{s}^{+}-\mu _{s}^{-}\in \mathcal{M}%
_{b}(\Omega ),$ where $\mu ^{0}=f- {div}g\in \mathcal{M}^{p}(\Omega ),$
and \ $\mu _{s}^{+},\mu _{s}^{-}$ are $p$-singular we say that a sequence $%
\left( \mu _{n}\right) $ is a \textbf{good approximation }of $\mu $ in $%
\mathcal{M}_{b}(\Omega )$ if it can be decomposed as
\begin{equation}
\mu _{n}=\mu _{n}^{0}+\lambda _{n}-\eta _{n},\quad \text{with}\quad \text{ }%
\mu _{n}^{0}=f_{n}- {div}g_{n},\quad \text{ }f_{n}\in L^{1}(\Omega
),\quad g_{n}\in (L^{p^{\prime }}(\Omega ))^{N},\quad \lambda _{n},\eta
_{n}\in \mathcal{M}_{b}^{+}(\Omega ),  \label{dol}
\end{equation}%
such that $(f_{n})$ converges to $f$ weakly in $L^{1}(\Omega ),$ $(g_{n})$
converges to $g$ strongly in $(L^{p^{\prime }}(\Omega ))^{N}$ and $( {div%
}g_{n})$ is bounded in $\mathcal{M}_{b}(\Omega ),$ and $(\rho _{n})$
converges to $\mu _{s}^{+}$ and $(\eta _{n})$ converges to $\mu _{s}^{-}$ in
the narrow topology.
\end{definition}

\begin{theorem}[\protect\cite{DMOP}]
\label{fund}Let $\mu \in \mathcal{M}_{b}(\Omega )$, and let $\left( \mu
_{n}\right) $ be a good approximation of $\mu $. Let $u_{n}$ be a R-solution
of
\begin{equation*}
-\Delta _{p}u_{n}=\mu _{n}\hspace{0.5cm}\text{in }\Omega ,\qquad
u_{n}=0\quad \text{on }\partial \Omega .
\end{equation*}%
Then there exists a subsequence $(u_{\nu })$ converging a.e. in $\Omega $ to
a R-solution $u$ of problem (\ref{mu}). And $(T_{k}(u_{\nu }))$ converges to
$T_{k}(u)$ strongly in $W_{0}^{1,p}(\Omega ).$
\end{theorem}

\begin{remark}
As a consequence, for any measure $\mu \in \mathcal{M}_{b}(\Omega ),$ there
exists at least a solution of problem (\ref{mu}). Indeed, it is pointed out
in \cite{DMOP} that any measure $\mu \in \mathcal{M}_{b}(\Omega )$ can be
approximated by such a sequence: extending $\mu $ by $0$ to $\mathbb{R}^{N},$
one can take $g_{n}=g,$ $f_{n}=\rho _{n}\ast f,$ $\lambda _{n}=\rho _{n}\ast
\mu _{s}^{+},$ $\eta _{n}=\rho _{n}\ast \mu _{s}^{-},$ where $(\rho _{n})$
is a regularizing sequence; then $f_{n},\lambda _{n},\eta _{n}\in
C_{b}^{\infty }\left( \Omega \right) $. Notice that this approximation does
not respect the sign: $\mu \in \mathcal{M}_{b}^{+}(\Omega )$ does not imply
that $\mu _{n}\in \mathcal{M}_{b}^{+}(\Omega )$. \medskip
\end{remark}

In the sequel we precise the approximation property, still partially used in
\cite[Theorem 2.18]{HaBi} for problem (\ref{pbu}).

\begin{lemma}
\label{app}Let $\mu \in \mathcal{M}_{b}(\Omega ).$ Then

(i) there exists a sequence $\left( \mu _{n}\right) $ of good approximations
of $\mu ,$ such $\mu _{n}\in $ $W^{-1,p^{\prime }}(\Omega ),$ and $\mu
_{n}^{0}$ has a compact support in $\Omega $, $\lambda _{n},\eta _{n}\in
C_{b}^{\infty }\left( \Omega \right) ,\left( f_{n}\right) $ converges to $f$
strongly in $L^{1}(\Omega ),$ and
\begin{equation}
\left\vert \mu _{n}\right\vert (\Omega )\leqq 4\left\vert \mu \right\vert
(\Omega ),\quad \forall n\in \mathbb{N}  \label{unif}
\end{equation}%
Moreover, if $\mu \in \mathcal{M}_{b}^{+}(\Omega )$, then one can find the
approximation such that $\mu _{n}\in \mathcal{M}_{b}^{+}(\Omega )$ and $%
\left( \mu _{n}\right) $ is nondecreasing.

(ii) there exists another sequence $\left( \mu _{n}\right) $ of good
approximations of $\mu ,$ with , with $f_{n},g_{n}\in \mathcal{D}\left(
\Omega \right) $, $\lambda _{n},\eta _{n}\in C_{b}^{\infty }\left( \Omega
\right) $, such that $\left( f_{n}\right) $ converges to $f$ strongly in $%
L^{1}(\Omega ),$ satisfying (\ref{unif}); if $\mu \in \mathcal{M}%
_{b}^{+}(\Omega )$, one can take $\mu _{n}^{0}\in \mathcal{D}^{+}\left(
\Omega \right) .$
\end{lemma}

\begin{proof}
(i) Let $\mu =\mu ^{0}+\mu _{s}^{+}-\mu _{s}^{-},$ where $\mu ^{0}\in
\mathcal{M}^{p}(\Omega ),$ \ $\mu _{s}^{+},\mu _{s}^{-}$ are $p$-singular
and $\mu _{1}=(\mu ^{0})^{+},\mu _{2}=(\mu ^{0})^{-};$ thus $\mu _{1}(\Omega
)+\mu _{2}(\Omega )+\mu _{s}^{+}(\Omega )+\mu _{s}^{-}(\Omega )$ $\leqq
2\left\vert \mu (\Omega )\right\vert .$ Following \cite{BoGaOr}, for $i=1,2,$
one has
\begin{equation*}
\mu _{i}=\varphi _{i}\gamma _{i},\qquad \text{with }\gamma _{i}\in \mathcal{M%
}_{b}^{+}(\Omega )\cap W^{-1,p^{\prime }}(\Omega )\text{ and }\varphi
_{i}\in L^{1}(\Omega ,\gamma _{i}).
\end{equation*}%
Let $(K_{n})_{n\geqq 1}$ be an increasing sequence of compacts of union $%
\Omega ;$ set
\begin{equation*}
\nu _{1,i}=T_{1}(\varphi _{i}\chi _{K_{1}})\gamma _{i},\quad \nu
_{n,i}=T_{n}(\varphi _{i}\chi _{K_{n}})\gamma _{i}-T_{n-1}(\varphi _{i}\chi
_{K_{n-1}})\gamma _{i},\quad \mu _{n,i}^{0}=\sum_{1}^{n}\nu
_{n,i}=T_{n}(\varphi _{i}\chi _{K_{n}})\gamma _{i}.
\end{equation*}%
Thus $\mu _{n,i}^{0}\in \mathcal{M}_{b}^{+}(\Omega )\cap W^{-1,p^{\prime
}}(\Omega ).$ Regularizing by $(\rho _{n}),$ there exists $\phi _{n,i}$ $\in
\mathcal{D}^{+}(\Omega )$ such that $\left\Vert \phi _{n,i}-\nu
_{n,i}\right\Vert _{W^{-1,p^{\prime }}(\Omega )}$ $\leqq 2^{-n}\mu
_{i}(\Omega ).$ Then $\xi _{n,i}=\sum_{1}^{n}\phi _{k,i}\in \mathcal{D}%
^{+}(\Omega );$ $(\eta _{n,i})$ converges strongly in $L^{1}(\Omega )$ to a
function $\xi _{i}$ and $\left\Vert \xi _{n,i}\right\Vert _{L^{1}(\Omega
)}\leqq \mu _{i}(\Omega ).$ Also setting
\begin{equation*}
G_{n,i}=\mu _{n,i}^{0}-\xi _{n,i}=\sum_{1}^{n}(\nu _{n,i}-\phi _{k,i})\in
W^{-1,p^{\prime }}(\Omega )\cap \mathcal{M}_{b}(\Omega ),
\end{equation*}%
then $\left( G_{n,i}\right) $ converges strongly in $W^{-1,p^{\prime
}}(\Omega )$ to some $G_{i}$, and $\mu _{i}=\xi _{i}+G_{i},$ and $\left\Vert
G_{n,i}\right\Vert _{\mathcal{M}_{b}(\Omega )}$ $\leqq 2\mu _{i}(\Omega ).$
Otherwise $\lambda _{n}=\rho _{n}\ast \mu _{s}^{+}$ and $\eta _{n}=\rho
_{n}\ast \mu _{s}^{-}$ $\in C_{b}^{\infty }\left( \Omega \right) $ converge
respectively to $\mu _{s}^{+},\mu _{s}^{-}$ in the narrow topology, with $%
\left\Vert \lambda _{n}\right\Vert _{L^{1}(\Omega )}\leqq \mu
_{s}^{+}(\Omega ),$ $\left\Vert \eta _{n}\right\Vert _{L^{1}(\Omega )}\leqq
\mu _{s}^{-}(\Omega ).$ Then we set
\begin{equation*}
\mu _{n}=\mu _{n}^{0}+\rho _{n}-\eta _{n}\qquad \text{with }\mu _{n}^{0}=\xi
_{n}+G_{n},\text{ }\xi _{n}=\xi _{n,1}-\xi _{n,2}\in \mathcal{D}(\Omega
),\quad G_{n}=G_{n,1}-G_{n,2}\in W^{-1,p^{\prime }}(\Omega )
\end{equation*}%
thus $\mu _{n}^{0}$ has a compact support. Moreover $\mu _{0}=\xi +G$ with $%
\xi =\xi _{1}-\xi _{2}\in \mathcal{D}(\Omega ),$ and $G=G_{1}-G_{2}=$ $%
\varphi + {div}g$ for some $\varphi \in L^{p^{\prime }}\left( \Omega
\right) $ and $g\in (L^{p^{\prime }}\left( \Omega \right) )^{N},$ and $%
\left( G_{n}\right) $ converges to $G$ in $W^{-1,p^{\prime }}(\Omega ).$
Then we can find $\psi _{n}\in L^{p^{\prime }}\left( \Omega \right) ,\phi
_{n}\in (L^{p^{\prime }}\left( \Omega \right) )^{N}$, such that $%
G_{n}-G=\psi _{n}+ {div}\phi _{n}$ and $\left\Vert G_{n}-G\right\Vert
_{W^{-1,p^{\prime }}(\Omega )}=\max (\left\Vert \psi _{n}\right\Vert
_{L^{p^{\prime }}\left( \Omega \right) },\left\Vert \phi _{n}\right\Vert
_{(L^{P^{\prime }}\left( \Omega \right) )^{N}})$; then $\mu _{0}=f+ {div}%
g$ with $f=\xi +\varphi $ and $\mu _{n}^{0}=f_{n}+ {div}g_{n},$ with $%
f_{n}=\xi _{n}+\varphi +\psi _{n},$ $g_{n}=g+\phi _{n}$. Thus $\left( \mu
_{n}\right) $ is a good approximation of $\mu ,$ and satisfies (\ref{unif}).
If $\mu $ is nonnegative, then $\mu _{n}$ is nonnegative.\medskip

(ii) We replace $\mu _{n}^{0}$ by $\rho _{m}\ast \mu _{n}^{0}$ $=\rho
_{m}\ast f_{n}+ {div}(\rho _{m}\ast g_{n})$, $m\in \mathbb{N},$ and
observe that $\left\vert \rho _{m}\ast \mu _{n}^{0}\right\vert (\Omega
)\leqq \left\vert \mu _{n}^{0}\right\vert (\Omega );$ then we can construct
another sequence satisfying the conditions.
\end{proof}

\subsection{\textbf{\ Locally renormalized solutions}}

Let $\mu \in \mathcal{M}(\Omega )$. Following the notion introduced in \cite%
{Bi1}, we say that $u$ is a \textbf{locally renormalized solution, }called%
\textbf{\ LR-solution}, of problem
\begin{equation}
-\Delta _{p}u=\mu ,\qquad \text{in }\Omega ,  \label{pli}
\end{equation}%
if $u$ is measurable and finite a.e. in $\Omega $, $T_{k}(u)\in
W_{loc}^{1,p}(\Omega )$ for any $k>0,$ and%
\begin{equation}
\left\vert u\right\vert ^{p-1}\in L_{loc}^{\sigma }(\Omega ),\forall \sigma
\in \left[ 1,N/(N-p\right) ;\qquad \left\vert \nabla u\right\vert ^{p-1}\in
L_{loc}^{\tau }(\Omega ),\forall \tau \in \tau \in \left[ 1,N/(N-1)\right) ,
\label{ugc}
\end{equation}%
and\textit{\ }for any\textit{\ }$h\in W^{1,\infty }(\mathbb{R})$\textit{\ }%
such that\textit{\ }$h^{\prime }$\textit{\ }has a compact support, and%
\textit{\ }$\varphi \in W^{1,m}(\Omega )$\textit{\ }for some\textit{\ }$m>N,$%
\textit{\ }with compact support, such that\textit{\ }$h(u)\varphi \in
W^{1,p}(\Omega ),$ there holds
\begin{equation}
\int_{\Omega }\left\vert \nabla u\right\vert ^{p-2}\nabla u.\nabla
(h(u)\varphi )dx=\int_{\Omega }h(u)\varphi d\mu _{0}+h(+\infty )\int_{\Omega
}\varphi d\mu _{s}^{+}-h(-\infty )\int_{\Omega }\varphi d\mu _{s}^{-}.
\label{plo}
\end{equation}

\begin{remark}
Hence the LR-solutions are solutions in $\mathcal{D}^{\prime }(\Omega ).$
From a recent result of \cite{KilKuTu}, if $\mu \in \mathcal{M}^{+}(\Omega )$%
, any $p$-superharmonic function is a LR-solution, and conversely any
LR-solution admits a $p$-superharmonic representant.
\end{remark}

\section{Existence in the subcritical case}

We first give a general existence result, where $H$ satisfies some
subcritical growth assumptions on $u$ and $\nabla u$, without any assumption
on the sign of $H$\ or $\mu $: we consider the problem%
\begin{equation}
-\Delta _{p}u+H(x,u,\nabla u)=\mu \hspace{0.5cm}\text{in }\Omega ,\qquad u=0%
\hspace{0.5cm}\text{on }\partial \Omega ,  \label{PR}
\end{equation}%
where $\mu \in \mathcal{M}_{b}(\Omega ).$ We say that $u$ is a R-solution of
problem (\ref{P0}) if  $T_{k}(u)\in W_{0}^{1,p}(\Omega )$ for any $k>0,$ and
$H(x,u,\nabla u)\in L^{1}(\Omega )$ and $u$ is a R-solution of
\begin{equation*}
-\Delta _{p}u=\mu -H(x,u,\nabla u),\qquad \text{in }\Omega ,\qquad u=0%
\hspace{0.5cm}\text{on }\partial \Omega .
\end{equation*}

\begin{theorem}
\label{subc}Let $\mu \in \mathcal{M}_{b}(\Omega ),$ and assume that
\begin{equation}
\left\vert H(x,u,\xi )\right\vert \leqq f(x)\left\vert u\right\vert
^{Q}+g(x)\left\vert \xi \right\vert ^{q}+\ell (x)  \label{fgh}
\end{equation}%
with $Q,q>0$ and $f\in L^{r}(\Omega )$ with $Qr^{\prime }<Q_{c},$ $g\in
L^{s}(\Omega )$ with $qs^{\prime }<q_{c}$, and $\ell \in L^{1}(\Omega
).\medskip $

Then there exists a R-solution of (\ref{PR}) if, either max($Q,q)>p-1$ and $%
\left\vert \mu \right\vert (\Omega )$ and $\left\Vert \ell \right\Vert
_{L^{1}(\Omega )}$ are small enough, or $q=p-1>Q$ and $\left\Vert
f\right\Vert _{L^{r}(\Omega )}$ is small enough, or $Q=p-1>q$ and $%
\left\Vert g\right\Vert _{L^{s}(\Omega )}$ is small enough, or $q,Q<p-1$.
\end{theorem}

\begin{proof}
\textbf{(i) Construction of a sequence of approximations.} We consider a
sequence $\left( \mu _{n}\right) _{n\geqq 1}$\textbf{\ }of good
approximations of\textbf{\ }$\mu $, given in Lemma \ref{app} (i). For any
fixed $n\in \mathbb{N}^{\ast },$ and any $v\in W_{0}^{1,p}(\Omega )$ we
define
\begin{equation*}
M(v)=\left\vert \Omega \right\vert ^{\frac{N-p}{N}-\frac{p-1}{Qr^{\prime }}%
}\left( \int_{\Omega }\left\vert v\right\vert ^{Qr^{\prime }}dx\right) ^{%
\frac{p-1}{Qr^{\prime }}}+\left\vert \Omega \right\vert ^{\frac{N-1}{N}-%
\frac{p-1}{qs^{\prime }}}\left( \int_{\Omega }\left\vert \nabla v\right\vert
^{qs^{\prime }}dx\right) ^{\frac{p-1}{qs^{\prime }}},
\end{equation*}%
\begin{equation*}
\Phi _{n}(v)(x)=-\frac{H(x,v(x),\nabla v(x))}{1+\frac{1}{n}(f(x)\left\vert
v(x)\right\vert ^{Q}+g(x)\left\vert \nabla v(x)\right\vert ^{q}+\ell (x))}
\end{equation*}%
so that $\left\vert \Phi _{n}(v)(x)\right\vert \leqq n$ a.e. in $\Omega .$
Let $\lambda >0$ be a parameter. Starting from $u_{1}\in W_{0}^{1,p}(\Omega
) $ such that $M(u_{1})\leqq \lambda ,$ we define $u_{2}\in
W_{0}^{1,p}(\Omega )$ as the solution of the problem
\begin{equation*}
-\Delta _{p}u_{2}=\Phi _{1}(u_{1})+\mu _{1}\hspace{0.5cm}\text{in }\Omega
,\qquad U_{2}=0\hspace{0.5cm}\text{on }\partial \Omega ,
\end{equation*}%
and by induction we define $u_{n}\in W_{0}^{1,p}(\Omega )$ as the solution
of
\begin{equation*}
-\Delta _{p}u_{n}=\Phi _{n-1}(u_{n-1})+\mu _{n}\hspace{0.5cm}\text{in }%
\Omega ,\qquad u_{n}=0\hspace{0.5cm}\text{on }\partial \Omega .
\end{equation*}%
From (\ref{sti}), for any $\sigma \in \left( 0,N/(N-p\right) $ and $\tau \in
\left( 0,N/(N-1)\right) ,$
\begin{equation*}
\left\vert \Omega \right\vert ^{\frac{N-p}{N}-\frac{1}{\sigma }%
}(\int_{\Omega }\left\vert u_{n}\right\vert ^{(p-1)\sigma }dx)^{\frac{1}{%
\sigma }}+\left\vert \Omega \right\vert ^{\frac{N-1}{N}-\frac{1}{\tau }%
}(\int_{\Omega }\left\vert \nabla u_{n}\right\vert ^{(p-1)\tau }dx)^{\frac{1%
}{\tau }}\leqq C(\int_{\Omega }\left\vert \Phi _{n-1}(u_{n-1})\right\vert
dx+4\left\vert \mu \right\vert (\Omega )),\text{ }
\end{equation*}%
with $C=C(N,p,\sigma ,\tau ).$ We take $\sigma =Qr^{\prime }/(p-1)$ and $%
\tau =qs^{\prime }/(p-1);$ since
\begin{equation}
\int_{\Omega }\left\vert H(x,u_{n-1},\nabla u_{n-1})\right\vert dx\leqq
\left\Vert f\right\Vert _{L^{r}(\Omega )}\int_{\Omega }\left\vert
u_{n-1}\right\vert ^{Qr^{\prime }}dx)^{\frac{1}{r^{\prime }}}+\left\Vert
g\right\Vert _{L^{s}(\Omega )}(\int_{\Omega }\left\vert \nabla
u_{n-1}\right\vert ^{qs^{\prime }}dx)^{\frac{1}{s^{\prime }}}+\left\Vert
\ell \right\Vert _{L^{1}(\Omega )}  \label{jil}
\end{equation}%
we obtain
\begin{equation*}
M(u_{n})\leqq C(\int_{\Omega }\left\vert H(x,u_{n-1},\nabla
u_{n-1})\right\vert dx+4\left\vert \mu \right\vert (\Omega ))\leqq
b_{1}M(u_{n-1})^{Q/(p-1)}+b_{2}M(u_{n-1})^{q/(p-1)}+\eta +a
\end{equation*}%
with $C=C(N,p,q,Q),$ and $b_{1}=C\left\Vert f\right\Vert _{L^{r}(\Omega
)}\left\vert \Omega \right\vert ^{\frac{1}{r^{\prime }}-\frac{Q}{Q_{c}}},$ $%
b_{2}=C\left\Vert g\right\Vert _{L^{s}(\Omega )}\left\vert \Omega
\right\vert ^{1/s^{\prime }-q/q_{c}},$ $\eta =C\left\Vert \ell \right\Vert
_{L^{1}(\Omega )},$ $a=4C\left\vert \mu \right\vert (\Omega ).$ Then by
induction, $M(u_{n})\leqq \lambda $ for any $n\geqq 1$ if%
\begin{equation}
b_{1}\lambda ^{Q/(p-1)}+b_{2}\lambda ^{q/(p-1)}+\eta +a\leqq \lambda .
\label{flo}
\end{equation}%
When $Q<p-1$ and $q<p-1,$ (\ref{flo}) holds for $\lambda $ large enough. In
the other cases, we note that it holds as soon as
\begin{equation}
b_{1}\lambda ^{Q/(p-1)-1}+b_{2}\lambda ^{q/(p-1)-1}\leqq 1/2,\text{\quad and
}\eta \leqq \lambda /4,\text{ }a\leqq \lambda /4.  \label{flou}
\end{equation}%
First suppose that $Q>p-1$ or $q>p-1.$ We take $\lambda \leqq 1,$ small
enough so that $(b_{1}^{Q/(p-1)}+b_{2}^{q/(p-1)})\lambda ^{\max
(Q,q)/(p-1)-1}\leqq 1/2,$ and then $\eta ,a\leqq \lambda /4$. Next suppose
for example that $Q=p-1>q,$ $a$ is arbitrary. If $b_{1}$ small enough, and $%
\eta ,a$ are arbitrary, then we obtain (\ref{flou}) for $\lambda $ large
enough.\medskip

\noindent \textbf{(ii) Convergence: }Since $M(u_{n})\leqq \lambda ,$ in turn
from (\ref{jil}), $(H(x,u_{n},\nabla u_{n}))$ is bounded in $L^{1}(\Omega ),$
and then also $\Phi _{n}(u_{n})$. Thus
\begin{equation*}
\int_{\Omega }\left\vert \Phi _{n-1}(u_{n-1})\right\vert dx+\left\vert \mu
_{n}\right\vert (\Omega )\leqq C_{\lambda }:=b_{1}\lambda
^{Q/(p-1)}+b_{2}\lambda M^{q/(p-1)}+\eta +4\left\vert \mu \right\vert
(\Omega ).
\end{equation*}%
From Lemma \ref{ldmop}, up to a subsequence, $(u_{n})$ converges $a.e.$ to a
function $u,$ $\left( \nabla u_{n}\right) $ converges $a.e.$ to $\nabla u,$
and $\left( u_{n}^{p-1}\right) $ converges strongly in $L^{\sigma }(\Omega
), $ for any $\sigma \in \left[ 1,N/(N-p)\right) ,$ and finally $\left(
\left\vert \nabla u_{n}\right\vert ^{p-1}\right) $ converges strongly in $%
L^{\tau }(\Omega ),$ for any $\tau \in \left[ 1,N/(N-1)\right) .$ Therefore $%
(u_{n}^{Qr^{\prime }})$ and $(\left\vert \nabla u_{n}\right\vert
^{qr^{\prime }})$ converge strongly in $L^{1}(\Omega ),$ in turn $\left(
\Phi _{n}(x,u_{n},\nabla u_{n})\right) $ converges strongly to $H(x,u,\nabla
u)$ in $L^{1}(\Omega ).$ Then $(\Phi _{n}(x,u_{n},\nabla u_{n})+\mu _{n})$
is a sequence of good approximations of $H(x,u,\nabla u)+\mu .$ From Theorem %
\ref{fund}, $u$ is a R-solution of problem (\ref{PR}).
\end{proof}

\begin{remark}
Our proof is not based on the Schauder fixed point theorem, so we do not
need that $1\leqq Qr^{\prime }$ or $1\leqq qs^{\prime }.$ Hence we improve
the former result of \cite{HaBi} for problem (\ref{pbu}) where $H$ only
depends on $u$, proved for $1\leqq Qr^{\prime },$ implying $1<Q_{c}.$ Here
we have no restriction on $Q_{c}$ and $q_{c}.$
\end{remark}

Next we consider the case where $H$ and $\mu $ are nonnegative; then we do
not need that the data are small:

\begin{theorem}
\label{posi}Consider the problem (\ref{PR})%
\begin{equation}
-\Delta _{p}u+H(x,u,\nabla u)=\mu \hspace{0.5cm}\text{in }\Omega ,\qquad u=0%
\hspace{0.5cm}\text{on }\partial \Omega ,  \label{P1}
\end{equation}%
where $\mu \in \mathcal{M}_{b}^{+}(\Omega ),$ and
\begin{equation}
0\leqq H(x,u,\xi )\leqq C(\left\vert u\right\vert ^{Q}+\left\vert \xi
\right\vert ^{q})+\ell (x),  \label{bic}
\end{equation}%
with $0<Q<Q_{c},0<q<q_{c},C>0,\ell \in L^{1}(\Omega ).$ Then there exists a
nonnegative R-solution of problem (\ref{P1}).
\end{theorem}

\begin{proof}
We use the good approximation of $\mu $ by a sequence of measures $\mu
_{n}=\mu _{n}^{0}+\lambda _{n},$ with $\mu _{n}^{0}\in \mathcal{D}^{+}\left(
\Omega \right) ,\lambda _{n}\in C_{b}^{+}(\Omega ),$ given at Lemma \ref{app}
(ii). Then there exists a weak nonnegative solution $u_{n}\in
W_{0}^{1,p}\left( \Omega \right) $ of the problem
\begin{equation*}
-\Delta _{p}u_{n}+H(x,u_{n},\nabla u_{n})=\mu _{n}\hspace{0.5cm}\text{in }%
\Omega ,\qquad u_{n}=0\quad \text{on }\partial \Omega .
\end{equation*}%
Indeed 0 is a subsolution, and the solution $\psi _{n}\in W_{0}^{1,p}(\Omega
)$ of $-\Delta _{p}\psi _{n}=\mu _{n}$ in $\Omega ,$ is a supersolution.
Since $\mu _{n}\in L^{\infty }(\Omega ),$ there holds $\psi \in C^{1,\alpha
}(\overline{\Omega })$ for some $\alpha \in \left( 0,1\right) $, thus $\psi
\in W^{1,\infty }(\Omega ).$ From \cite[Theorem 2.1]{BoMuPu}, since $%
Q_{c}\leqq p,$ there exists a weak solution $u_{n}\in W_{0}^{1,p}(\Omega ),$
such that $0\leq u_{n}\leq \psi _{n},$ hence $u_{n}\in L^{\infty }(\Omega ),$
and $u_{n}\in W_{loc}^{1,r}(\Omega )$ for some $r>p.$ Taking $\varphi
=k^{-1}T_{k}(u_{n}-m)$ with $m\geq 0,$ $k>0,$ as a test function, we get
from (\ref{unif})
\begin{equation}
\frac{1}{k}\int_{\left\{ m\leqq u\leqq m+k\right\} }\left\vert \nabla
u_{n}\right\vert ^{p}dx\leq \mu _{n}(\Omega )\leq 4\mu (\Omega ),
\label{xyz}
\end{equation}%
then from Lemma \ref{ldmop}, up to a subsequence, $\left( u_{n}\right) $
converges $a.e.$ to a function $u,$ $\left( T_{k}(u_{n})\right) $ converges
weakly in $W_{0}^{1,p}(\Omega )$and $\left( \nabla u_{n}\right) $ converges $%
a.e.$ to $\nabla u,$ and $\left( \left\vert \nabla u_{n}\right\vert
^{p}\right) ,$ $\left( u_{n}^{p-1}\right) $ converges strongly in $L^{\sigma
}(\Omega )$ for any $\sigma \in \left[ 1,N/(N-p)\right) ,$ $\left(
\left\vert \nabla u_{n}\right\vert ^{p-1}\right) $ converges strongly in $%
L^{\tau }(\Omega ),$ for any $\tau \in \left[ 1,N/(N-1)\right) .$ Then $%
(u_{n}^{Qr^{\prime }})$ and $(\left\vert \nabla u_{n}\right\vert
^{qr^{\prime }})$ converge strongly in $L^{1}(\Omega ),$ in turn $\left(
H(x,u_{n},\nabla u_{n})\right) $ converges strongly to $H(x,u,\nabla u)$ in $%
L^{1}(\Omega ).$ Applying Theorem \ref{fund} to $\mu _{n}-H(x,u_{n},\nabla
u_{n})$ as above, we still obtain that $u$ is a R-solution of (\ref{P1}).
\end{proof}

\section{Necessary conditions for existence and removability results}

Let $\mu \in \mathcal{M}(\Omega ).$ We consider the local solutions of
\begin{equation}
-\Delta _{p}u+H(x,u,\nabla u)=\mu \qquad \text{in }\Omega ,  \label{eqmu}
\end{equation}%
We say that $u$ is a \textbf{weak solution} of (\ref{eqmu}) if $u$ is
measurable and finite a.e. in $\Omega $, $T_{k}(u)\in W_{loc}^{1,p}(\Omega )$
for any $k>0,$ $H(x,u,\nabla u)\in L_{loc}^{1}(\Omega )$ and (\ref{eqmu})
holds in $\mathcal{D}^{\prime }(\Omega ).$ We say that $u$ is a \textbf{%
LR-solution }of (\ref{eqmu}) if $T_{k}(u)\in W_{loc}^{1,p}(\Omega )$ for any
$k>0,$ and $\left\vert \nabla u\right\vert ^{q}\in L_{loc}^{1}(\Omega )$ and
$u$ is a LR-solution of%
\begin{equation*}
-\Delta _{p}u=\mu -H(x,u,\nabla u),\qquad \text{in }\Omega .
\end{equation*}%
\textit{\ }

\begin{remark}
\label{gre} If $q\geqq 1$ and $u$ is a weak solution, then $u$ satisfies (%
\ref{ugc}), see for example \cite[Lemma 2.2 and 2.3]{Le}, thus $u\in
W_{loc}^{1,q}(\Omega ).$
\end{remark}

\begin{lemma}
\label{much}Let $\mu \in \mathcal{M}(\Omega ).$ Assume that (\ref{eqmu})
admits a weak solution $u.$

(i) If
\begin{equation}
\left\vert H(x,u,\xi )\right\vert \leqq C_{1}\left\vert \xi \right\vert
^{q}+\ell (x)  \label{mis}
\end{equation}%
with $C_{1}>0$ and $\ell \in L^{1}(\Omega ),$ then setting $%
C_{2}=C_{1}+q_{\ast }-1$, for any $\zeta \in \mathcal{D}^{+}(\Omega ),$
\begin{equation}
\left\vert \int_{\Omega }\zeta ^{q_{\ast }}d\mu \right\vert \leq
C_{2}\int_{\Omega }\left\vert \nabla u\right\vert ^{q}\zeta ^{q_{\ast
}}dx+\int_{\Omega }\left\vert \nabla \zeta \right\vert ^{q_{\ast
}}dx+\int_{\Omega }\ell \zeta ^{q_{\ast }}dx.  \label{fir}
\end{equation}

(ii) If $H$ has a constant sign, and
\begin{equation}
C_{0}\left\vert \xi \right\vert ^{q}-\ell (x)\leqq \left\vert H(x,u,\xi
)\right\vert ,  \label{mus}
\end{equation}%
then for some $C=C(C_{0},p,q),$%
\begin{equation}
\int_{\Omega }\left\vert \nabla u\right\vert ^{q}\zeta ^{q_{\ast }}dx\leqq
C(\left\vert \int_{\Omega }\zeta ^{q_{\ast }}d\mu \right\vert +\int_{\Omega
}\left\vert \nabla \zeta \right\vert ^{q_{\ast }}dx+\int_{\Omega }\ell \zeta
^{q_{\ast }}dx)  \label{sec}
\end{equation}
\end{lemma}

\begin{proof}
By density, we can take $\zeta ^{q_{\ast }}$ as a test function, and get%
\begin{equation*}
\int_{\Omega }\zeta ^{q_{\ast }}d\mu =-\int_{\Omega }H(x,u,\nabla u)\zeta
^{q_{\ast }}dx+q_{\ast }\int_{\Omega }\left\vert \nabla u\right\vert
^{p-2}\nabla u.\zeta ^{q_{\ast }-1}\nabla \zeta dx;
\end{equation*}%
and from the H\"{o}lder inequality, for any $\varepsilon >0,$
\begin{equation}
q_{\ast }\int_{\Omega }\left\vert \nabla u\right\vert ^{p-1}\zeta ^{q_{\ast
}-1}\left\vert \nabla \zeta \right\vert dx\leqq (q_{\ast }-1)\varepsilon
\int_{\Omega }\left\vert \nabla u\right\vert ^{q}\zeta ^{q_{\ast
}}dx+\varepsilon ^{1-q_{\ast }}\int_{\Omega }\left\vert \nabla \zeta
\right\vert ^{q_{\ast }}dx  \label{eps}
\end{equation}%
which implies (\ref{fir}). If $H$ has a constant sign, then
\begin{eqnarray*}
C_{0}\int_{\Omega }\left\vert \nabla u\right\vert ^{q}\zeta ^{q_{\ast
}}dx-\int_{\Omega }\ell dx &\leqq &\int_{\Omega }\left\vert H(x,u,\nabla
u)\right\vert \zeta ^{q_{\ast }}dx=\left\vert \int_{\Omega }H(x,u,\nabla
u)\zeta ^{q_{\ast }}dx\right\vert  \\
&\leqq &\left\vert \int_{\Omega }\zeta ^{q_{\ast }}d\mu \right\vert +q_{\ast
}\int_{\Omega }\left\vert \nabla u\right\vert ^{p-1}\zeta ^{q_{\ast
}-1}\left\vert \nabla \zeta \right\vert dx,
\end{eqnarray*}%
thus (\ref{sec}) follows after taking $\varepsilon $ small enough.
\end{proof}

\begin{proposition}
\label{nec}Let $\mu \in \mathcal{M}(\Omega ),$ and assume that (\ref{eqmu})
admits a weak solution $u.\medskip $

(i) If (\ref{mis}) holds, then $\mu \in \mathcal{M}^{q_{\ast }}(\Omega
).\medskip $

(ii) If $H(x,u,\xi )\leqq -C_{0}\left\vert \xi \right\vert ^{q}$ and $\mu $
and $u$ are nonnegative, then in addition there exists $C=C(C_{0},p,q)>0$
such that for any compact $K\subset \Omega ,$
\begin{equation}
\mu (K)\leqq Ccap_{1,q_{\ast }}(K,\Omega ).  \label{capa}
\end{equation}
\end{proposition}

\begin{proof}
(i) Let $E$ be a Borel set such that $cap_{1,q_{\ast }}(E,\Omega )=0.$ There
exist two measurable disjoint sets $A,B$ such that $\Omega =A\cup B$ and $%
\mu ^{+}(B)=\mu ^{-}(A)=0.$ Let us show that $\mu ^{+}(A\cap E)=0.$ Let $K$
be any fixed compact set in $A\cap E.$ Since $\mu ^{-}(K)=0,$ for any $%
\delta >0$ there exists a regular domain $\omega \subset \subset \Omega $
containing $K,$ such that $\mu ^{-}(\omega )<\delta .$ Then there exists $%
\zeta _{n}\in \mathcal{D}(\omega )$ such that $0\leq \zeta _{n}\leq 1,$ and $%
\zeta _{n}=1$ on a neighborhood of $K$ contained in $\omega ,$ and $\left(
\zeta _{n}\right) $ converges to in $W^{1,q_{\ast }}(\mathbb{R}^{N})$ and $%
a.e.$ in $\Omega $, see \cite{AdPo}. There holds
\begin{equation*}
\mu ^{+}(K)\leq \int_{\omega }\zeta _{n}^{q_{\ast }}d\mu ^{+}=\int_{\omega
}\zeta _{n}^{q_{\ast }}d\mu +\int_{\omega }\zeta _{n}^{q_{\ast }}d\mu
^{-}\leq \int_{\omega }\zeta _{n}^{q_{\ast }}d\mu +\delta
\end{equation*}%
and from (\ref{fir}),%
\begin{equation*}
\left\vert \int_{\Omega }\zeta _{n}^{q_{\ast }}d\mu \right\vert \leq
C_{2}\int_{\Omega }\left\vert \nabla u\right\vert ^{q}\zeta _{n}^{q_{\ast
}}dx+\int_{\Omega }\left\vert \nabla \zeta _{n}\right\vert ^{q_{\ast
}}dx+\int_{\Omega }\ell \zeta _{n}^{q_{\ast }}dx
\end{equation*}%
And $\lim_{n\rightarrow \infty }\int_{\Omega }\left\vert \nabla u\right\vert
^{q}\zeta _{n}^{q_{\ast }}dx=0,$ from the dominated convergence theorem,
thus $\left\vert \int_{\Omega }\zeta _{n}^{q_{\ast }}d\mu \right\vert \leq
\delta $ for large $n;$ then $\mu ^{+}(K)\leq 2\delta $ for any $\delta >0,$
thus $\mu ^{+}(K)=0,$ hence $\mu ^{+}(A\cap E)=0;$ similarly we get $\mu
^{-}(B\cap E)=0,$ hence $\mu (E)=0.$

\noindent (ii) Here we find
\begin{equation*}
\int_{\Omega }\zeta ^{q_{\ast }}d\mu +C_{0}\int_{\Omega }\left\vert \nabla
u\right\vert ^{q}\zeta ^{q_{\ast }}dx\leqq q_{\ast }\int_{\Omega }\left\vert
\nabla u\right\vert ^{p-2}\nabla u.\zeta ^{q_{\ast }-1}\nabla \zeta dx
\end{equation*}%
and hence from (\ref{eps}) with $\varepsilon >0$ small enough, for some $%
C=C(C_{0},p,q),$
\begin{equation*}
\int_{\Omega }\zeta ^{q_{\ast }}d\mu \leq C\int_{\Omega }\left\vert \nabla
\zeta \right\vert ^{q_{\ast }}dx
\end{equation*}%
and (\ref{capa}) follows, see \cite{Maz}.
\end{proof}

\begin{remark}
Property (ii) extends the results of \cite{HaMaVe} and \cite[Theorem 3.1]{Ph}
for equation (\ref{sou}).
\end{remark}

Next we show a removability result:

\begin{theorem}
\label{remo}Assume that $H$ has a constant sign and satisfies (\ref{mis})
and (\ref{mus}). Let $F$ be any relatively closed subset of $\Omega ,$ such
that $cap_{1,q_{\ast }}(F,\mathbb{R}^{N})=0$, and $\mu \in \mathcal{M}%
^{q_{\ast }}(\Omega )$.

(i) Let $1<q\leqq p.$ Let $u$ be any LR-solution of
\begin{equation}
-\Delta _{p}u+H(x,u,\nabla u)=\mu \qquad \text{in }\Omega \backslash K
\label{dra}
\end{equation}%
Then $u$ is a LR-solution of%
\begin{equation}
-\Delta _{p}u+H(x,u,\nabla u)=\mu \qquad \text{in }\Omega .  \label{dre}
\end{equation}%
(ii) Let $q>p$ and $u$ be a weak solution of (\ref{dra}), then $u$ is a weak
solution of (\ref{dre}).
\end{theorem}

\begin{proof}
(i) Let $1<q\leqq p.$ From our assumption, $T_{k}(u)\in W_{loc}^{1,p}(\Omega
\backslash F),$ for any $k>0,$ and $\left\vert u\right\vert ^{p-1}\in
L_{loc}^{\sigma }(\Omega ),$ for any $\sigma \in \left[ 1,N/(N-p)\right) ,$
and $\left\vert \nabla u\right\vert ^{p-1}\in L_{loc}^{\tau }(\Omega
\backslash F),$ for any $\tau \in \left[ 1,N/(N-1)\right) ,$ and $\left\vert
\nabla u\right\vert ^{q}\in L_{loc}^{1}(\Omega \backslash F).$ For any
compact $K\subset \Omega $, there holds $cap_{1,p}(F\cap K,\mathbb{R}^{N})=0$%
, because $p\leqq q_{\ast },$ thus $T_{k}(u)\in W_{loc}^{1,p}(\Omega ),$ see
\cite[Theorem 2.44]{HeKilMa}. And $u$ is measurable on $\Omega $ and finite
a.e. in $\Omega ,$ thus we can define $\nabla u$ a.e. in $\Omega $ by the
formula $\nabla u(x)=\nabla T_{k}(u)(x)$ a.e. on the set $\left\{ x\in
\Omega :\left\vert u(x)\right\vert \leqq k\right\} .\medskip $

Let us consider a fixed function $\zeta \in \mathcal{D}^{+}(\Omega )$ and
let $\omega \subset \subset \Omega $ such that supp$\zeta \subset \omega $
and set $K_{\varsigma }=F\cap $ supp$\zeta .$ Then $K_{\varsigma }$ is a
compact and $cap_{1,q_{\ast }}(K,\mathbb{R}^{N})=0$. Thus there exists $%
\zeta _{n}\in \mathcal{D}(\omega )$ such that $0\leq \zeta _{n}\leq 1,$ and $%
\zeta _{n}=1$ on a neighborhood of $K$ contained in $\omega ,$ and $\left(
\zeta _{n}\right) $ converges to $0$ in $W^{1,q_{\ast }}(\mathbb{R}^{N});$
we can assume that the convergence holds everywhere on $\mathbb{R}%
^{N}\backslash N,$ where $cap_{1,q_{\ast }}(N,\mathbb{R}^{N})=0,$ see for
example \cite[Lemmas 2.1,2.2]{BaPi}. From Lemma \ref{much} applied to $\xi
_{n}=\zeta (1-\zeta _{n})$ in $\Omega \backslash F,$ we have
\begin{eqnarray}
\int_{\Omega }\left\vert \nabla u\right\vert ^{q}\xi _{n}^{q_{\ast }}dx
&\leqq &C(\int_{\Omega }\xi _{n}^{q_{\ast }}d\left\vert \mu \right\vert
+\int_{\Omega }\left\vert \nabla \xi _{n}\right\vert ^{q_{\ast
}}dx+\int_{\Omega }\ell \xi _{n}^{q_{\ast }}dx)  \notag \\
&\leqq &C(\int_{\Omega }\zeta ^{q_{\ast }}d\left\vert \mu \right\vert
+\int_{\Omega }\left\vert \nabla \zeta \right\vert ^{q_{\ast
}}dx+\int_{\Omega }\left\vert \nabla \zeta _{n}\right\vert ^{q_{\ast
}}dx+\int_{\Omega }\ell \zeta ^{q_{\ast }}dx).  \label{vol}
\end{eqnarray}%
From the Fatou Lemma, we get $\left\vert \nabla u\right\vert ^{q}\zeta
^{q_{\ast }}\in L^{1}(\Omega )$ and
\begin{equation}
\int_{\Omega }\left\vert \nabla u\right\vert ^{q}\zeta ^{q_{\ast }}dx\leqq
C_{\zeta }:=C(\int_{\Omega }\zeta ^{q_{\ast }}d\left\vert \mu \right\vert
+\int_{\Omega }\left\vert \nabla \zeta \right\vert ^{q_{\ast
}}dx\int_{\Omega }\ell \zeta ^{q_{\ast }}dx),  \label{eli}
\end{equation}%
where $C_{\zeta }$ also depends on $\zeta .$ Taking  $T_{k}(u)\xi
_{n}^{q_{\ast }}$ as test function, we obtain
\begin{eqnarray*}
&&\int_{\Omega }\left\vert \nabla (T_{k}(u))\right\vert ^{p}\xi
_{n}^{q_{\ast }}dx+\int_{\Omega }H(x,u,\nabla u)T_{k}(u)\xi _{n}^{q_{\ast
}}dx \\
&=&\int_{\Omega }T_{k}(u)\xi _{n}^{q_{\ast }}d\mu _{0}+k(\int_{\Omega }\xi
_{n}^{q_{\ast }}d\mu _{s}^{+}+\int_{\Omega }\xi _{n}^{q_{\ast }}d\mu
_{s}^{-})+\int_{\Omega }T_{k}(u)\left\vert \nabla u\right\vert ^{p-2}\nabla
u.\nabla (\xi _{n}^{q_{\ast }})dx;
\end{eqnarray*}%
From the H\"{o}lder inequality, we deduce
\begin{eqnarray*}
&&\frac{1}{k}\left\vert \int_{\Omega }T_{k}(u)\left\vert \nabla u\right\vert
^{p-2}\nabla u.\nabla (\xi _{n}^{q_{\ast }})dx\right\vert  \\
&\leqq &q_{\ast }(\int_{\Omega }\zeta ^{q_{\ast }-1}\left\vert \nabla
u\right\vert ^{p-1}\left\vert \nabla \zeta \right\vert )dx+\int_{\Omega
}\zeta ^{q_{\ast }}\left\vert \nabla u\right\vert ^{p-1}\left\vert \nabla
\zeta _{n}\right\vert dx) \\
&\leqq &(q_{\ast }-1)\int_{\Omega }\left\vert \nabla u\right\vert ^{q}\zeta
^{q_{\ast }}dx+\int_{\Omega }\left\vert \nabla \zeta \right\vert ^{q_{\ast
}}dx+q_{\ast }(\int_{\Omega }\left\vert \nabla u\right\vert ^{q}\zeta
^{q_{\ast }}dx+\int_{\Omega }\zeta ^{q_{\ast }}\left\vert \nabla \zeta
_{n}\right\vert ^{q_{\ast }}dx) \\
&\leqq &2q_{\ast }C_{\zeta }+\int_{\Omega }\left\vert \nabla \zeta
\right\vert ^{q_{\ast }}dx+o(n).
\end{eqnarray*}%
Thus from (\ref{mis}), with a new constant $C_{\zeta },$
\begin{equation*}
\int_{\Omega }\left\vert \nabla (T_{k}(u))\right\vert ^{p}\xi _{n}^{q_{\ast
}}dx\leqq (k+1)C_{\zeta }+o(n);
\end{equation*}%
hence from the Fatou Lemma,
\begin{equation*}
\int_{\Omega }\left\vert \nabla (T_{k}(u))\right\vert ^{p}\zeta ^{q_{\ast
}}dx\leqq (k+1)C_{\zeta }.
\end{equation*}%
Therefore $\left\vert u\right\vert ^{p-1}\in L_{loc}^{\sigma }(\Omega ),$ $%
\forall $ $\sigma \in \left[ 1,N/(N-p)\right) $ and $\left\vert \nabla
u\right\vert ^{p-1}\in L_{loc}^{\tau }(\Omega ),\forall $ $\tau \in \left[
1,N/(N-1)\right) ,$ from a variant of the estimates of \cite{Be} and \cite%
{BoGa1}, see \cite[Lemma 3.1]{PePoPor}.\medskip

Finally we show that $u$ is a LR-solution in $\Omega :$ let\textit{\ }$h\in
W^{1,\infty }(\mathbb{R})$\textit{\ }such that\textit{\ }$h^{\prime }$%
\textit{\ }has a compact support\textit{, }and\textit{\ }$\varphi \in
W^{1,m}(\Omega )$\textit{\ }for some\textit{\ }$m>N,$\textit{\ }with compact
support in $\Omega $, such that\textit{\ }$h(u)\varphi \in W^{1,p}(\Omega );$
let $\omega \subset \subset \Omega $ such that supp$\zeta \subset \omega $
and set $K=F\cap $supp$\zeta ,$ and consider $\zeta _{n}\in \mathcal{D}(%
\mathbb{R}^{N})$ as above; then $(1-\zeta _{n})\varphi \in W^{1,m}(\Omega
\backslash F)$ and $h(u)(1-\zeta _{n})\varphi \in W^{1,p}(\Omega \backslash
F)$ and has a compact support in $\Omega \backslash F,$ then we can write
\begin{equation*}
I_{1}+I_{2}+I_{3}+I_{4}=\int_{\Omega }h(u)\varphi (1-\zeta _{n})d\mu
_{0}+h(+\infty )\int_{\Omega }\varphi (1-\zeta _{n})d\mu _{s}^{+}-h(-\infty
)\int_{\Omega }\varphi (1-\zeta _{n})d\mu _{s}^{-},
\end{equation*}%
with
\begin{eqnarray*}
&&I_{1}=\int_{\Omega }\left\vert \nabla u\right\vert ^{p-2}\nabla
u.h^{\prime }(u)\varphi (1-\zeta _{n})dx,\qquad I_{2}=-\int_{\Omega
}\left\vert \nabla u\right\vert ^{p-2}\nabla u.h(u)\varphi \nabla \zeta
_{n}dx \\
&&I_{3}=\int_{\Omega }\left\vert \nabla u\right\vert ^{p-2}\nabla
u.h(u)(1-\zeta _{n})\nabla \varphi dx,\qquad I_{4}=\int_{\Omega
}H(x,u,\nabla u)h(u)\varphi (1-\zeta _{n})dx.
\end{eqnarray*}%
We can go to the limit in $I_{1}$ as $n\rightarrow \infty ,$ from the
dominated convergence theorem, since there exists $a>0$ such that
\begin{equation*}
\int_{\Omega }\left\vert \nabla u\right\vert ^{p-2}\nabla u.h^{\prime
}(u)\varphi (1-\zeta _{n})dx=\int_{\Omega }\left\vert \nabla
T_{a}(u)\right\vert ^{p-2}\nabla T_{a}(u).h^{\prime }(T_{a}(u))\varphi
(1-\zeta _{n})dx.
\end{equation*}%
Furthermore $I_{2}=o(n),$ because%
\begin{equation*}
\left\vert \int_{\Omega }\left\vert \nabla u\right\vert ^{p-2}\nabla
u.h(u)\varphi \nabla \zeta _{n}dx\right\vert \leqq \left\Vert h\right\Vert
_{L^{\infty }(\mathbb{R})}(\int_{\Omega }\left\vert \nabla u\right\vert
^{q}\varphi dx)^{1/q}\left\Vert \nabla \zeta _{n}\right\Vert _{L^{q_{\ast }}(%
\mathbb{R}^{N})};
\end{equation*}%
we can go to the limit in $I_{3}$ because $\left\vert \nabla \varphi
\right\vert \in L^{m}(\Omega )$ and $\left\vert \nabla u\right\vert
^{p-1}\in L_{loc}^{\tau }(\Omega ),\forall \tau \in \left[ 1,N/(N-1)\right) ;
$ in $I_{4}$ from (\ref{eli}) and (\ref{mis}), and in the right hand side
because $h(u)\varphi \in L^{1}(\Omega ,d\mu _{0}),$ see \cite[Remark 2.26]%
{DMOP} and $\zeta _{n}\rightarrow 0$ everywhere in $\mathbb{R}^{N}\backslash
N$ and $\mu (N)=0.$ Then we  conclude:
\begin{eqnarray*}
&&\int_{\Omega }\left\vert \nabla u\right\vert ^{p-2}\nabla u.\nabla
(h(u)\varphi )dx+\int_{\Omega }H(x,u,\nabla u)h(u)\varphi dx \\
&=&\int_{\Omega }h(u)\varphi d\mu _{0}+h(+\infty )\int_{\Omega }\varphi d\mu
_{s}^{+}-h(-\infty )\int_{\Omega }\varphi d\mu _{s}^{-}.
\end{eqnarray*}

(ii) Assume that $q>p>1$ (hence $1<q_{\ast }<p)$ and $u$ is a weak solution
in $\Omega \backslash F.$ Then $u\in W_{loc}^{1,q}(\Omega \backslash F)$
implies $u\in W_{loc}^{1,q_{\ast }}(\Omega \backslash F)=W_{loc}^{1,q_{\ast
}}(\Omega ),$ hence $\left\vert \nabla u\right\vert $ is well defined in $%
L_{loc}^{1}(\Omega ).$ As in part (i) we obtain that $\left\vert \nabla
u\right\vert ^{q}\zeta ^{q_{\ast }}\in L^{1}(\Omega ),$ hence $\left\vert
\nabla u\right\vert ^{q}\in L_{loc}^{1}(\Omega ).$ For any $\varphi \in
\mathcal{D}(\Omega ),$ and $\omega $ containing supp$\varphi ,$ we have $%
\varphi (1-\zeta _{n})\in \mathcal{D}(\Omega \backslash F),$ then we can
write $J_{1}+J_{2}+J_{3}=\int_{\Omega }\varphi (1-\zeta _{n})d\mu ,$ with
\begin{equation*}
J_{1}=\int_{\Omega }(1-\zeta _{n})\left\vert \nabla u\right\vert
^{p-2}\nabla u.\nabla \varphi dx,\;J_{2}=-\int_{\Omega }\varphi \left\vert
\nabla u\right\vert ^{p-2}\nabla u.\nabla \zeta _{n}dx,\;J_{3}=\int_{\Omega
}H(x,u,\nabla u)\varphi (1-\zeta _{n})dx.
\end{equation*}%
Now we can go to the limit in $J_{1}$ and $J_{3}$ from the dominated
convergence theorem, because $\left\vert \nabla u\right\vert ^{q}\in
L_{loc}^{1}(\Omega )$ and $q>p-1;$ and $(\int_{\Omega }\varphi (1-\zeta
_{n})d\mu )$ converges to $\int_{\Omega }\varphi d\mu $ as above. And $J_{2}$
converges to $0,$ because $\left\vert \nabla u\right\vert ^{p-1}\in
L_{loc}^{q/(p-1)}(\Omega )$ and $\left\vert \nabla \zeta _{n}\right\vert $
tends to $0$ in $L^{q_{\ast }}(\Omega ).$ Then $u$ is a weak solution in $%
\Omega .$
\end{proof}

\section{Existence in the supercritical case}

Here the problem is delicate and many problems are still unsolved.

\subsection{Case of a source term}

Here we consider problem
\begin{equation}
-\Delta _{p}u=\left\vert \nabla u\right\vert ^{q}+\mu \hspace{0.5cm}\text{in
}\Omega ,\qquad u=0\hspace{0.5cm}\text{on }\partial \Omega .  \label{P2}
\end{equation}%
The main question is the following:

\textit{If }$\mu \in M_{b}^{q_{\ast }}(\Omega )$\textit{\ satisfies
condition (\ref{capa}) with a constant }$C>0$\textit{\ small enough, does (%
\ref{P2}) admit a solution?\medskip }

In the case $p=2<q,$ the problem has been solved in \cite{HaMaVe}. In that
case one can define the solutions in a very weak sense. According to \cite%
{BrCMR}, setting $\rho (x)=dist(x,\partial \Omega ),$ a function $u$ is
called a \textit{very weak solution} of (\ref{P2}) if $u\in
W_{loc}^{1,q}\left( \Omega \right) \cap L^{1}(\Omega ),$ $\left\vert \nabla
u\right\vert ^{q}\in L^{1}(\Omega ,\rho dx)$ and for any $\varphi \in
C^{2}\left( \overline{\Omega }\right) $ such that $\varphi =0$ on $\partial
\Omega ,$
\begin{equation*}
-\int_{\Omega }u\Delta \varphi dx=\int_{\Omega }\left\vert \nabla
u\right\vert ^{q}\varphi dx+\int_{\Omega }\varphi d\mu .
\end{equation*}

\begin{theorem}[\protect\cite{HaMaVe}]
\label{Ve}Let $\mu \in \mathcal{M}^{+}(\Omega ).$ If $1<q$ and $p=2$ and (%
\ref{P2}) has a very weak solution, then
\begin{equation}
\mu (K)\leqq Ccap_{1,q^{\prime }}(K,\Omega )  \label{clo}
\end{equation}
for any compact $K\subset \Omega ,$ and some $C<C_{1}(N,q).$ Conversely, if $%
2<q$ and (\ref{clo}) holds for some $C<C_{2}(N,q,\Omega )$ then (\ref{P2})
has a very weak nonnegative solution.
\end{theorem}

In the general case $p>1,$ such a notion of solution does not exist. The
problem (\ref{P2}) with $p<q$ was studied by \cite{Ph} for signed measures $%
\mu \in \mathcal{M}_{b}(\Omega )$ such that
\begin{equation*}
\left[ \mu \right] _{1,q^{\ast },\Omega }=\sup \left\{ \frac{\left\vert \mu
(K\cap \Omega )\right\vert }{Cap_{1,q^{\ast }}(K,\mathbb{R}^{N})}:K\text{
compact of }\mathbb{R}^{N},Cap_{1,q^{\ast }}(K,\mathbb{R}^{N})>0\right\}
<\infty .
\end{equation*}

\begin{theorem}[\protect\cite{Ph}]
\label{phuc}Let $1<p<q.$ Let $\mu \in \mathcal{M}_{b}(\Omega ).$ There
exists $C_{1}=C_{1}(N,p,q,\Omega )$ such that if
\begin{equation}
\left\vert \mu (K\cap \Omega )\right\vert \leqq Ccap_{1,q^{\ast }}(K,\mathbb{%
R}^{N})  \label{ploc}
\end{equation}%
for any compact $K\subset \mathbb{R}^{N},$ and some $C<C_{1}$, then (\ref{P2}%
) has a weak solution $u\in W_{0}^{1,q}(\Omega ),$ such that $\left[
\left\vert \nabla u\right\vert ^{q}\right] _{1,q^{\ast },\Omega }$ is
finite. In particular this holds for any $\mu \in L^{N/q^{\ast },\infty
}(\Omega ).$
\end{theorem}

Very recently the case $p=q$, has been studied in \cite{JaMaVe} for signed
measures satisfying a trace inequality: setting $p^{\#}=(p-1)^{2-p}$ if $%
p\geqq 2,$ $p^{\#}=1$ if $p<2,$ they show in particular the following:

\begin{theorem}[\protect\cite{JaMaVe}]
\label{jmv}Let $1<p=q.$ Let $\mu \in \mathcal{M}_{b}(\Omega )$ such that
\begin{equation}
-C_{1}\int_{\Omega }\left\vert \nabla \zeta \right\vert ^{p}dx\leqq
\int_{\Omega }\left\vert \zeta \right\vert ^{p}d\mu \leq C_{2}\int_{\Omega
}\left\vert \nabla \zeta \right\vert ^{p}dx,\qquad \forall \zeta \in
\mathcal{D}(\Omega ),
\end{equation}%
with $C_{1}>0$ and $C_{2}\in (0,p^{\#}).$ Then (\ref{P2}) has a weak
solution $u\in W_{loc}^{1,p}(\Omega ).$
\end{theorem}

The existence for problem (\ref{P2}) is still open in the case $q<p$ for $%
p\neq 2$

\subsection{Case of an absorption term}

Here we consider problem (\ref{P0}) in case of absorption, where $\mu \in
M_{b}^{+}(\Omega )$ and we look for a nonnegative solution. In the model
case
\begin{equation}
-\Delta _{p}u+\left\vert \nabla u\right\vert ^{q}=\mu \hspace{0.5cm}\text{in
}\Omega ,\qquad u=0\quad \text{on }\partial \Omega ,  \label{P4}
\end{equation}%
the main question is the following: \textit{If }$\mu \in M_{b}^{q_{\ast
}+}(\Omega ),$\textit{\ hence }$\mu =f+ {div}g,$\textit{\ with }$f\in
L^{1}\left( \Omega \right) $\textit{\ and }$g\in (L^{q/(p-1)}\left( \Omega
\right) )^{N}$\textit{, does (\ref{P4}) admits a nonnegative solution?}

\begin{remark}
Up to changing $u$ into $-u,$ the results of Theorem \ref{phuc} and \ref{jmv}
are also available for the problem (\ref{P4}) but we have \textbf{no
information on the sign of} $u.$
\end{remark}

In the sequel we give two partial results of existence.

\subsubsection{Case $q\leqq p$ and $\protect\mu \in \mathcal{M}%
_{b}^{p+}(\Omega )$\protect\medskip}

Here we assume that $\mu \in \mathcal{M}_{b}^{p+}(\Omega ),$ subspace of $%
\mathcal{M}_{b}^{q_{\ast }+}(\Omega ).$ Our proof is directly inspired from
the results of \cite{BoGaOr} for the problem (\ref{P1}), where $q=p$ and $%
H(x,u,\xi )u\geqq 0.$

\begin{theorem}
\label{one}Let $p-1<q\leqq p.$ Let $\mu \in \mathcal{M}_{b}^{p+}(\Omega ),$
and
\begin{eqnarray}
0 &\leqq &H(x,u,\xi )\leqq C_{1}\left\vert \xi \right\vert ^{p}+\ell (x),
\label{jin} \\
H(x,u,\xi ) &\geqq &C_{0}\left\vert \xi \right\vert ^{q}\text{ for }u\geqq L,
\label{jan}
\end{eqnarray}%
with $\ell (x)\in L^{1}(\Omega ),C_{k},C_{0},L\geqq 0$. Then there exists a
nonnegative R-solution of problem (\ref{P0}).
\end{theorem}

\begin{remark}
The result was known in the case where $H(x,u,\nabla u)=\left\vert \nabla
u\right\vert ^{q}$, $p=2,$ and $\mu \in L^{1}\left( \Omega \right) $ (see
for example \cite{AbPePr}, where the existence for any $\mu \in \mathcal{M}%
_{b}^{2+}(\Omega )$ is also claimed, without proof). For $p\neq 2,$ the case
$q<p,$ $\mu \in L^{1}\left( \Omega \right) $ is partially treated in \cite%
{PerPr}.
\end{remark}

\begin{proof}
Let $\mu =f- {div}g$ with $f\in L^{1+}(\Omega )$ and $g=(g_{i})\in
(L^{p^{\prime }}(\Omega ))^{N}.$ Here again we use the good approximation of
$\mu $ by a sequence of measures $\mu _{n}\in \mathcal{M}_{b}^{+}(\Omega )$
given at Lemma \ref{app} (ii), $\lambda _{n}=0,$ thus $\mu _{n}=\mu
_{n}^{0}=f_{n}- {div}g_{n},$ with $f_{n}\in \mathcal{D}^{+}\left( \Omega
\right) $ and $g_{n}=(g_{n,i})\in (\mathcal{D}\left( \Omega \right) ^{N}).$
Hence there exists a weak nonnegative solution $u_{n}\in W_{0}^{1,p}\left(
\Omega \right) $ of the problem
\begin{equation*}
-\Delta _{p}u_{n}+H(x,u_{n},\nabla u_{n})=\mu _{n}\hspace{0.5cm}\text{in }%
\Omega ,\qquad u_{n}=0\quad \text{on }\partial \Omega .
\end{equation*}%
Since $H(x,u,\xi )\geqq 0,$ taking $\varphi =k^{-1}T_{k}(u_{n}-m)$ with $%
m\geq 0,$ $k>0,$ as a test function, we still obtain (\ref{xyz}). From Lemma %
\ref{ldmop}, up to a subsequence, $(u_{n})$ converges $a.e.$ to a function $%
u,$ $\left( T_{k}(u_{n})\right) $ converges weakly in $W_{0}^{1,p}\left(
\Omega \right) $, $\left( \nabla u_{n}\right) $ converges $a.e.$ to $\nabla
u,$ and $\left( u_{n}^{p-1}\right) $ converges strongly in $L^{\sigma
}(\Omega ),$ for any $\sigma \in \left[ 1,N/(N-p)\right) .$ Thus $%
\lim_{k\rightarrow \infty }\sup_{n\in \mathbb{N}}\left\vert \left\{
u_{n}>k\right\} \right\vert =0,$ and $\left( \left\vert \nabla
u_{n}\right\vert ^{p-1}\right) $ converges strongly in $L^{\tau }(\Omega ),$
for any $\tau \in \left[ 1,N/(N-1)\right) .$ Moreover the choice of $\varphi
$ with $m+k>L$ gives
\begin{equation*}
\frac{1}{k}\int_{\left\{ m\leqq u\leqq m+k\right\} }\left\vert \nabla
u_{n}\right\vert ^{p}dx+C_{0}\int_{\left\{ u_{n}\geq m+k\right\} }\left\vert
\nabla u_{n}\right\vert ^{q}dx\leqq \mu _{n}(\Omega )\leqq 4\mu (\Omega ).
\end{equation*}%
Taking $m=0$ we obtain
\begin{equation*}
\int_{\Omega }\left\vert \nabla u_{n}\right\vert ^{q}dx\leq \int_{\left\{
u_{n}\geq k\right\} }\left\vert \nabla u_{n}\right\vert ^{q}dx+\int_{\Omega
}\left\vert \nabla T_{k}(u_{n})\right\vert ^{q}dx\leq 4C_{0}^{-1}\mu (\Omega
)+\int_{\Omega }\left\vert \nabla T_{k}(u_{n})\right\vert ^{p}dx+\left\vert
\Omega \right\vert
\end{equation*}%
since $q\leqq p;$ thus from the Fatou Lemma, $\left\vert \nabla u\right\vert
^{q}\in L^{1}\left( \Omega \right) .$ Moreover using $\varphi
=T_{1}(u_{n}-k),$
\begin{equation*}
\int_{\left\{ k-1\leqq u_{n}\leqq k\right\} }\left\vert \nabla
u_{n}\right\vert ^{p}dx+\int_{\left\{ u_{n}\geq k\right\} }H(x,u_{n},\nabla
u_{n})dx\leqq \int_{\left\{ u_{n}\geq k-1\right\} }f_{n}dx+\int_{\left\{
k-1\leq u_{n}\leq k\right\} }\left\vert g_{n}.\nabla u_{n}\right\vert dx.
\end{equation*}%
Therefore, from the H\"{o}lder inequality,
\begin{equation*}
\int_{\left\{ k-1\leq u_{n}\leq k\right\} }\left\vert \nabla
u_{n}\right\vert ^{p}dx+p^{\prime }\int_{\left\{ u_{n}\geq k\right\}
}H(x,u_{n},\nabla u_{n})dx\leq \int_{\left\{ u_{n}\geq k-1\right\}
}f_{n}dx+(\sum_{i=1}^{N}\int_{\left\{ k-1\leq u_{n}\leq k\right\}
}\left\vert g_{n,i}\right\vert ^{p^{\prime }}dx).
\end{equation*}%
From Lemma \ref{app}, there holds
\begin{equation}
\lim_{k\rightarrow \infty }\sup_{n\in \mathbb{N}}(\int_{\left\{ k-1\leq
u_{n}\leq k\right\} }\left\vert \nabla u_{n}\right\vert ^{p}dx+\int_{\left\{
u_{n}\geq k\right\} }H(x,u_{n},\nabla u_{n})dx)=0.  \label{rou}
\end{equation}%
\medskip

Next we prove the strong convergence of the truncates in $W_{0}^{1,p}\left(
\Omega \right) $ as in \cite{BoGaOr}: we take as test function
\begin{equation*}
\varphi _{n}=\Phi (T_{k}(u_{n})-T_{k}(u)),\text{ where }\Phi (s)=se^{\theta
^{2}s^{2}/4},
\end{equation*}%
where $\theta >0$ will be chosen after, thus $\Phi ^{\prime }(s)\geq \theta
\left\vert \Phi (s)\right\vert +1/2.$ Then $\varphi _{n}\in
W_{0}^{1,p}\left( \Omega \right) \cap L^{\infty }\left( \Omega \right) ,$
and we have $\left\vert \varphi _{n}\right\vert \leq \Phi (k);$ setting $%
\psi _{n}=\Phi ^{\prime }(T_{k}(u_{n})-T_{k}(u)),$ we have $0\leq \psi
_{n}\leq \Phi ^{\prime }(k).$ Then $\varphi _{n}\rightarrow 0,$ $\psi
_{n}\rightarrow 1$ in $L^{\infty }\left( \Omega \right) $ weak * and $a.e.$
in $\Omega .$ We set $a(\xi )=\left\vert \xi \right\vert ^{p-2}\xi ,$ and
\begin{equation*}
X=\int_{\Omega }(a(\nabla (T_{k}(u_{n}))-a(\nabla (T_{k}(u))).\nabla
(T_{k}(u_{n})-T_{k}(u))\psi _{n}dx,
\end{equation*}%
and get
\begin{equation*}
X+I_{1}=I_{2}+I_{3}+I_{4},
\end{equation*}
with
\begin{equation*}
I_{1}=\int_{\Omega }H(x,u_{n},\nabla u_{n})\varphi _{n}dx,\qquad
I_{2}=\int_{\Omega }a(\nabla (T_{k}(u)).\nabla (T_{k}(u)-T_{k}(u_{n}))\psi
_{n}dx,
\end{equation*}%
\begin{equation*}
I_{3}=\int_{\Omega }f_{n}\varphi _{n}dx+\int_{\Omega } {div}%
(g_{n}-g)\varphi _{n}dx+\int_{\Omega }g.\nabla (T_{k}(u_{n})-T_{k}(u))\psi
_{n}dx,
\end{equation*}%
\begin{equation*}
I_{4}=-\int_{\Omega }a(\nabla (u_{n}-T_{k}(u_{n})).\nabla
(T_{k}(u_{n})-T_{k}(u))\psi _{n}dx=\int_{\left\{ u_{n}\geq k\right\}
}a(\nabla (u_{n}-T_{k}(u_{n})).\nabla (T_{k}(u))\psi _{n}dx.
\end{equation*}%
One can easily see that $\left\vert I_{2}\right\vert +\left\vert
I_{3}\right\vert +$ $\left\vert I_{4}\right\vert =o(n)$. Since $%
H(x,u_{n},\nabla u_{n})\geqq 0$ for $u_{n}\geq k,$ then $X\leqq I_{5}+o(n),$
where
\begin{equation*}
I_{5}=\left\vert \int_{\left\{ u_{n}<k\right\} }H(x,u_{n},\nabla
u_{n})\varphi _{n}dx\right\vert \leq C_{1}\int_{\Omega }\left\vert \nabla
(T_{k}u_{n})\right\vert ^{p})\left\vert \varphi _{n}\right\vert
dx+\int_{\Omega }l\left\vert \varphi _{n}\right\vert dx\leqq
C_{1}(Y+I_{7})+o(n),
\end{equation*}%
with
\begin{equation*}
Y=\int_{\Omega }(a(\nabla (T_{k}(u_{n}))-a(\nabla (T_{k}(u))).\nabla
(T_{k}(u_{n})-T_{k}(u))\left\vert \varphi _{n}\right\vert dx,
\end{equation*}%
\begin{equation*}
I_{7}=\int_{\Omega }a(\nabla (T_{k}(u))).\nabla
(T_{k}(u_{n})-T_{k}(u))\left\vert \varphi _{n}\right\vert dx+\int_{\Omega
}(a(\nabla (T_{k}(u_{n})).\nabla (T_{k}(u))\left\vert \varphi
_{n}\right\vert dx
\end{equation*}%
and then $I_{7}=o(n).$ We get finally $X\leqq C_{1}Y+o(n);$ choosing $\theta
=2C_{1},$ we deduce that
\begin{equation*}
\int_{\Omega }(a(\nabla (T_{k}(u_{n}))-a(\nabla (T_{k}(u))).\nabla
(T_{k}(u_{n})-T_{k}(u))dx=o(n).
\end{equation*}%
Hence $\left( T_{k}(u_{n})\right) $ converges strongly to $T_{k}(u)$ in $%
W_{0}^{1,p}\left( \Omega \right) .$ Therefore $H(x,u_{n},\nabla u_{n})$ is
equi-integrable, from (\ref{jin}) and (\ref{rou}), since for any measurable
set $E\subset \Omega ,$
\begin{equation*}
\int_{E}H(x,u_{n},\nabla u_{n})dx\leqq C_{1}\int_{E}\left\vert \nabla
(T_{k}u_{n})\right\vert ^{p}dx+\int_{E}\ell dx+\int_{\left\{ u_{n}\geq
k\right\} }H(x,u_{n},\nabla u_{n})dx.
\end{equation*}%
\medskip Then $(H(x,u_{n},\nabla u_{n}))$ converges to $H(x,u,\nabla u)$
strongly in $L^{1}(\Omega );$ thus $(\mu _{n}-H(x,u_{n},\nabla u_{n}))$ is a
good approximation of $\mu -H(x,u,\nabla u),$ and $u$ is a R-solution of
problem (\ref{PR}) from Theorem \ref{fund}.
\end{proof}

\begin{remark}
In the case $p-1<q<p,$ and if (\ref{jin}) is replaced by
\begin{equation}
0\leqq H(x,u,\xi )\leqq C_{1}\left\vert \xi \right\vert ^{q}+\ell (x),
\end{equation}%
the proof is much shorter: in order to prove the equi-integrability of $%
\left( H(x,u_{n},\nabla u_{n})\right) $ we do not need to prove the strong
convergence of the truncates: indeed for any measurable set $E\subset \Omega
,$
\begin{equation*}
\int_{E}\left\vert \nabla u_{n}\right\vert ^{q}dx\leqq \int_{E}\left\vert
\nabla (T_{k}u_{n})\right\vert ^{q}dx+\int_{\left\{ u_{n}\geq k\right\}
}\left\vert \nabla u_{n}\right\vert ^{q}dx
\end{equation*}%
and $\left( \nabla T_{k}(u_{n})\right) $ converges strongly to $\nabla
T_{k}(u)$ in $L^{q}\left( \Omega \right) $ and (\ref{rou}) holds. Then $%
\left( H(x,u_{n},\nabla u_{n})\right) $ converges to $H(x,u,\nabla u)$
strongly in $L^{1}\left( \Omega \right) .$
\end{remark}

\subsubsection{Case where $\protect\mu $ satisfies (\protect\ref{capa})}

Here we assume that $\mu \in \mathcal{M}_{b}^{+}(\Omega )$ satisfies a
capacity condition of type (\ref{capa}). For simplicity we assume that $\mu $
has a compact support in $\Omega .$ In the sequel we prove the following:

\begin{theorem}
\label{two}Let $1<q\leqq p$ or $p=2.$ Assume that $\mu \in \mathcal{M}%
_{b}^{+}(\Omega ),$ has a compact support and satisfies
\begin{equation}
\mu (K)\leqq C_{1}cap_{1,q\ast }(K,\Omega ),\text{ \quad for any compact }%
K\subset \Omega ,  \label{dac}
\end{equation}%
for some $C_{1}=C(N,q,\Omega )>0$ (non necessarily small). Then there exists
a nonnegative R-solution $u$ of problem (\ref{P4}), such that $\left[
\left\vert \nabla u\right\vert ^{q}\right] _{1,q^{\ast },\Omega }$ is finite.
\end{theorem}

First recall some equivalent properties of measures, see \cite[Theorem 1.2]%
{MazVe}, \cite[Lemma 3.3]{HaMaVe}, see also \cite{Ph}:

\begin{remark}
\label{equi}1) Let $\mu \in \mathcal{M}_{b}^{+}(\Omega ),$ extended by $0$
to $\mathbb{R}^{N}.$ Then (\ref{dac}) holds if and only if there exists $%
C_{2}>0$ such that
\begin{equation}
\int_{\Omega }\zeta ^{q_{\ast }}d\mu \leq C_{2}\int_{\Omega }\left\vert
\nabla \zeta \right\vert ^{q_{\ast }}dx,\qquad \forall \zeta \in \mathcal{D}%
^{+}(\Omega );  \label{doc}
\end{equation}
the constants of equivalence between $C_{1},C_{2}$ only depend on $N,q_{\ast
},\Omega .$

If moreover $\mu $ has a compact support $K_{0}\subset \Omega ,$ then (\ref%
{dac}) holds if and only if there exists $C_{3}>0$ such that
\begin{equation}
\mu (K)\leqq C_{3}Cap_{1,q_{\ast }}(K,\mathbb{R}^{N})\qquad \text{for any
compact }K\subset \mathbb{R}^{N};  \label{duc}
\end{equation}
the constants of equivalence between $C_{1},C_{3}$ only depend on $N,q_{\ast
},K_{0}$.

2) Let $\nu \in \mathcal{M}_{b}^{+}(\mathbb{R}^{N}).$ Then (\ref{duc}) holds
if and only if there exists $C_{4}>0$ such that $J_{1}(\nu )$ is finite $%
a.e. $ and
\begin{equation}
J_{1}((J_{1}(\nu ))^{q_{\ast }})\leqq C_{4}J_{1}(\nu )\qquad a.e.\text{ in }%
\mathbb{R}^{N};  \label{jj}
\end{equation}
the constants of equivalence between $C_{3},C_{4}$ do not depend on $\nu .$
\end{remark}

Following the ideas of \cite[Theorem 3.4]{Ph} we prove a convergence Lemma:

\begin{lemma}
\label{conv} Let $\left( z_{n}\right) $ be a sequence of nonnegative
functions, converging $a.e.$ in $L^{1}(\Omega ).$ Extending $z_{n}$ by $0$
in $\mathbb{R}^{N}\backslash \Omega ,$ assume that for some $C>0,$
\begin{equation*}
\int_{\Omega }z_{n}^{\frac{q}{p-1}}\xi ^{q_{\ast }}dx\leqq C\int_{\Omega
}\left\vert \nabla \xi \right\vert ^{q_{\ast }}dx\qquad \forall n\in \mathbb{%
N},\forall \xi \in \mathcal{D}^{+}(\mathbb{R}^{N}).
\end{equation*}%
Then $\left( z_{n}\right) $ converges strongly in $L^{q/(p-1)}(\Omega )$.
\end{lemma}

\begin{proof}
From our assumption, $\left( z_{n}\right) $ is bounded in $%
L^{q/(p-1)}(\Omega ),$ then up to a subsequence, it converges to some $z$
weakly in $L^{q/(p-1)}(\Omega )$ and $a.e.$ in $\Omega .$ Consider a ball $%
B\supset \Omega $ of radius $2diam\Omega $, and denote by $G$ the Green
function associated to $-\Delta $ in $B.$ Set $w_{n}=z_{n}^{q/(p-1)},$ and
extend $w_{n}$ by $0$ to $\mathbb{R}^{N}\backslash \Omega .$ Then for any
compact $K\subset \mathbb{R}^{N},$
\begin{equation*}
\int_{K\cap \Omega }w_{n}dx=\int_{K\cap B}w_{n}dx\leqq CCap_{1,q\ast }(K,%
\mathbb{R}^{N}),
\end{equation*}%
which means that $\left[ w_{n}\right] _{1,q^{\ast },B}$ is bounded, and
\begin{equation*}
\left\vert \nabla G(w_{n})(x)\right\vert \leqq \int_{B}\left\vert \nabla
_{x}G(x,y)\right\vert w_{n}(y)dy\leqq CG_{1}\ast w_{n}(x),
\end{equation*}%
with $C=C(N,$diam$\Omega ).$ In turn from \cite[Corollary 2.5]{Ph}, we get
the upperestimate
\begin{equation*}
\left[ \left\vert \nabla G(w_{n})\right\vert ^{\frac{q}{p-1}}\right]
_{1,q^{\ast },B}\leqq C\left[ \left\vert G_{1}\ast w_{n}\right\vert ^{\frac{q%
}{p-1}}\right] _{1,q^{\ast },B}\leqq C\left[ w_{n}\right] _{1,q^{\ast
},B}^{q/(p-1)}.
\end{equation*}
Therefore $\left( \left\vert \nabla G(w_{n})\right\vert \right) $  is
bounded in $L^{q/(p-1)}(B),$ thus $\left( \left\vert \nabla
G(w_{n}-w)\right\vert \right) $ is bounded in $L^{q/(p-1)}(B).$ Let $\varphi
\in \mathcal{D}\left( B\right) $ and $\varepsilon >0$ be fixed. Since $%
(z_{n})$ converges $a.e.$ to $z$ $,$ from the Egoroff theorem, there exists
a measurable set $\omega _{\varepsilon }\subset B$ such that $\left(
w_{n}\right) $ converges to $w=z^{q/(p-1)}$ uniformly on $\omega
_{\varepsilon },$ and $\left\Vert \left\vert \nabla \varphi \right\vert
\right\Vert _{L^{q^{\ast }}(B\backslash \omega _{\varepsilon })}\leqq
\varepsilon .$ There holds%
\begin{eqnarray*}
\left\vert \int_{\Omega }(w_{n}-w)\varphi dx\right\vert  &=&\left\vert
\int_{B}(w_{n}-w)\varphi dx\right\vert = \\
&=&\left\vert -\int_{B}(\Delta (G(w_{n}-w)\varphi dx\right\vert =\left\vert
-\int_{B}\nabla (G(w_{n}-w).\nabla \varphi dx\right\vert
\end{eqnarray*}%
Considering the two integrals on $B\backslash \omega _{\varepsilon }$ and $%
\omega _{\varepsilon }$ we find $\lim \int_{\Omega }(w_{n}-w)\varphi dx=0.$
Taking $\varphi =1$ on $\Omega ,$ it follows that $\lim \int_{\Omega
}z_{n}^{q/(p-1)}dx=\int_{\Omega }z^{q/(p-1)}dx$ and the proof is
done.\medskip
\end{proof}

\begin{proof}[Proof of Theorem \protect\ref{two}]
From our assumption, $\mu \in \mathcal{M}^{q^{\ast }}(\Omega )$. We consider
the problem associated to $\mu _{n}=\mu \ast \rho _{n}$
\begin{equation}
-\Delta _{p}u_{n}+\left\vert \nabla u_{n}\right\vert ^{q}=\mu _{n}\hspace{%
0.5cm}\text{in }\Omega ,\qquad u_{n}=0\quad \text{on }\partial \Omega .
\label{Pn}
\end{equation}%
For $q\leqq p,$ from \cite[Theorem 2.1]{BoMuPu}, as in the proof of Theorem %
\ref{posi}, (\ref{Pn}) admits a nonnegative solution $u_{n}\in $ $%
W_{0}^{1,p}(\Omega )\cap C^{1,\alpha }(\overline{\Omega })$. Moreover we can
approximate $u_{n}$ in $C^{1,\alpha }(\overline{\Omega })$ by the solution $%
u_{n,\varepsilon }$ ($\varepsilon >0)$ of the problem
\begin{equation*}
- {div}((\varepsilon ^{2}+\left\vert \nabla u_{n,\varepsilon
}\right\vert ^{2})^{\frac{p-2}{2}}\nabla u_{n,\varepsilon })+(\varepsilon
^{2}+\left\vert \nabla u_{n,\varepsilon }\right\vert ^{2})^{\frac{q}{2}}=\mu
_{n}\hspace{0.5cm}\text{in }\Omega ,\qquad u_{n,\varepsilon }=0\quad \text{%
on }\partial \Omega .
\end{equation*}%
Multiplying this equation by $\xi ^{q_{\ast }}$ with $\xi \in \mathcal{D}%
^{+}(\mathbb{R}^{N}),$ we obtain
\begin{eqnarray*}
&&q_{\ast }\int_{\Omega }(\varepsilon ^{2}+\left\vert \nabla
u_{n,\varepsilon }\right\vert ^{2})^{\frac{p-2}{2}}\nabla u_{n,\varepsilon
}.\xi ^{q_{\ast }-1}\nabla \xi dx+\int_{\Omega }(\varepsilon ^{2}+\left\vert
\nabla u_{n,\varepsilon }\right\vert ^{2})^{\frac{q}{2}}\xi ^{q_{\ast }}dx \\
&=&\int_{\Omega }\xi ^{q_{\ast }}\mu _{n}dx+q_{\ast }\int_{\partial \Omega
}\xi ^{q_{\ast }}(\varepsilon ^{2}+\left\vert \nabla u_{n,\varepsilon
}\right\vert ^{2})^{\frac{p-2}{2}}\nabla u_{n,\varepsilon }.\nu ds.
\end{eqnarray*}%
The boundary term is nonpositive, hence going to the limit as $\varepsilon
\rightarrow 0,$ we get
\begin{equation}
\int_{\Omega }\left\vert \nabla u_{n}\right\vert ^{q}\xi ^{q_{\ast }}dx\leqq
\int_{\Omega }\xi ^{q_{\ast }}\mu _{n}dx+q_{\ast }\int_{\Omega }\left\vert
\nabla u_{n}\right\vert ^{p-2}\nabla u_{n}.\xi ^{q_{\ast }-1}\nabla \xi dx
\label{ineg}
\end{equation}%
When $p=2,$ existence also holds for $q>2$, from \cite{Li}; and then $%
u_{n}\in C^{2}\left( \overline{\Omega }\right) $, thus (\ref{ineg}) is still
true. As in Lemma \ref{much}, it follows that for any $\xi \in \mathcal{D}%
^{+}(\mathbb{R}^{N})$
\begin{equation}
\int_{\Omega }\left\vert \nabla u_{n}\right\vert ^{q}\xi ^{q_{\ast }}dx\leqq
C(\int_{\Omega }\xi ^{q_{\ast }}d\mu _{n}+\int_{\Omega }\left\vert \nabla
\xi \right\vert ^{q_{\ast }}dx).  \label{tri}
\end{equation}%
Otherwise, since $\mu _{n}(\Omega )\leqq \mu (\Omega ),$ from Lemma \ref%
{ldmop}, up to a subsequence $(u_{n})$ converges $a.e.$ to a function $u,$ $%
\left( T_{k}(u_{n})\right) $ converges weakly in $W_{0}^{1,p}(\Omega )$ and $%
\left( \nabla u_{n}\right) $ converges $a.e.$ to $\nabla u$ in $\Omega .$
Note also that $\left( \mu _{n}\right) $ is a sequence of good
approximations of $\mu ,$ since $\mu $ has a compact support (see \cite%
{BiHuVe}).\ From (\ref{sec}), for any $\xi \in \mathcal{D}^{+}(\mathbb{R}%
^{N}),$ we have $\lim \int_{\Omega }\xi ^{q_{\ast }}d\mu _{n}=\int_{\Omega
}\xi ^{q_{\ast }}d\mu ,$ since $\xi ^{q_{\ast }}\in C_{c}(\mathbb{R}^{N}).$
Then $\int_{\Omega }\xi ^{q_{\ast }}d\mu \leqq C\int_{\Omega }\left\vert
\nabla \xi \right\vert ^{q_{\ast }}dx.$ From the Fatou Lemma, we obtain%
\begin{equation}
\int_{\Omega }\left\vert \nabla u\right\vert ^{q}\xi ^{q_{\ast }}dx\leqq
C(\int_{\Omega }\xi ^{q_{\ast }}d\mu +\int_{\Omega }\left\vert \nabla \xi
\right\vert ^{q_{\ast }}dx)\leqq C\int_{\Omega }\left\vert \nabla \xi
\right\vert ^{q_{\ast }}dx,  \label{lip}
\end{equation}%
hence $\left\vert \nabla u\right\vert ^{q}\in L^{1}\left( \Omega \right) .$
And then for any compact $K\subset \mathbb{R}^{N},$ taking $\xi =1$ on $K$,
\begin{equation*}
\int_{K\cap \Omega }\left\vert \nabla u\right\vert ^{q}dx\leqq CCap_{1,q\ast
}(K,\mathbb{R}^{N}),
\end{equation*}%
thus $\left[ \left\vert \nabla u\right\vert ^{q}\right] _{1,q^{\ast },\Omega
}$ is finite. Moreover,\textbf{\ }extending\textbf{\ }$\mu $ by $0$ to $%
\mathbb{R}^{N}\backslash \Omega ,$ we see from Remark \ref{equi} that $\mu $
satisfies condition (\ref{doc}), which is equivalent to (\ref{jj}). By
convexity, $\mu _{n}$\textbf{\ }also satisfies (\ref{jj}) and hence (\ref%
{doc}), with the same constants, $i.e.$ for any $n\in \mathbb{N}$ and any $%
\xi \in \mathcal{D}^{+}(\mathbb{R}^{N}),$
\begin{equation}
\int_{\Omega }\xi ^{q_{\ast }}d\mu _{n}\leqq C_{2}\int_{\Omega }\left\vert
\nabla \xi \right\vert ^{q_{\ast }}dx  \label{hyp}
\end{equation}%
Then from (\ref{tri}) with another $C>0,$
\begin{equation}
\int_{\Omega }\left\vert \nabla u_{n}\right\vert ^{q}\xi ^{q_{\ast }}dx\leqq
C\int_{\Omega }\left\vert \nabla \xi \right\vert ^{q_{\ast }}dx  \label{lop}
\end{equation}%
Next we can apply Lemma \ref{conv} to $z_{n}=\left\vert \nabla
u_{n}\right\vert ^{p-1},$ since $\left( \nabla u_{n}\right) $ converges $%
a.e. $ to $\nabla u$ in $\Omega .$ Then $\left( \left\vert \nabla
u_{n}\right\vert ^{q}\right) $ converges strongly in $L^{1}(\Omega )$ to $%
\left\vert \nabla u\right\vert ^{q}.$ Thus $(\mu _{n}-\left\vert \nabla
u_{n}\right\vert ^{q})$ is a good approximation of $(\mu -\left\vert \nabla
u\right\vert ^{q}).$ From Theorem \ref{fund}, $u$ is a R-solution of the
problem.\medskip

From \cite[Theorem 1.4]{JaMaVe}, condition (\ref{lip}) (for $N\geqq 2)$
implies that $q_{\ast }<N,$ that means $q>q_{c},$ or $\left\vert \nabla
u\right\vert ^{q}=0$ in $\Omega ,$ thus $\mu =0.$ If $\mu = {div}g$ with
$g\in (L^{N(q+1-p)/(p-1),\infty }(\Omega ))^{N}$ with compact support, then $%
\left\vert g\right\vert ^{\frac{q}{p-1}}\in L^{N/q_{\ast },\infty }(\Omega
), $ thus
\begin{equation*}
\int_{\Omega }\zeta ^{q_{\ast }}\left\vert g\right\vert ^{\frac{q}{p-1}%
}dx\leq C_{2}\int_{\Omega }\left\vert \nabla \zeta \right\vert ^{q_{\ast
}}dx,\qquad \forall \zeta \in \mathcal{D}^{+}(\Omega ).
\end{equation*}%
Hence $\mu $ satisfies (\ref{doc}) from the H\"{o}lder inequality. Note that
$\mu \in \mathcal{M}^{q_{\ast }}(\Omega ),$ since $q>q_{c}$ implies $%
\left\vert g\right\vert \in L^{q/(p-1)}(\Omega )^{N}$.
\end{proof}

\begin{remark}
Let $q\leqq p$ and $\mu = {div}g,$ where $g$ has a compact support in $%
\Omega .$ From Theorems \ref{one} and \ref{two}, we have existence when $%
g\in (L^{p^{\prime }}(\Omega ))^{N},$ or when $g\in
(L^{N(q+1-p)/(p-1),\infty }(\Omega ))^{N}$. Observe that $L^{p^{\prime
}}(\Omega )\supset L^{N(q+1-p)/(p-1)}(\Omega )$ if and only if $\tilde{q}%
\leqq q,$ where $\tilde{q}$ is defined at (\ref{defi}). Hence Theorem \ref%
{one} brings better results than Theorem \ref{two} when $\tilde{q}\leqq
q\leqq p.$
\end{remark}

\begin{remark}
The extension of this result to the case $p<q,p\neq 2$ will be studied in a
further article.
\end{remark}

\section{Some regularity results}

In this section we give some regularity properties for the problem:
\begin{equation}
-\Delta _{p}u+H(x,u,\nabla u)=0\qquad \text{in }\Omega .  \label{equ}
\end{equation}

We first recall some local estimates of the gradient for renormalized
solutions, see \cite{HaBi}, following the first results of \cite{BoGa}, and
many others, see among them \cite{AlFeTr}, \cite{KiBo}.

\begin{lemma}
\label{boot} Let $u$ be the R-solution of problem
\begin{equation*}
-\Delta _{p}u=f\hspace{0.5cm}\text{in }\Omega ,\qquad u=0\quad \text{on }%
\partial \Omega ,
\end{equation*}%
with $f\in L^{m}(\Omega ),$ $1<m<N.$ Set $\overline{m}=Np/(Np-N+p)=p/\tilde{q%
},$where $\tilde{q}$ is defined in (\ref{defi}). $\medskip $

\noindent (i) If $m>N/p,$ then $u\in L^{\infty }(\Omega ).$ If $m=N/p,$ then
$u\in L^{k}(\Omega )$ for any $k\geqq 1.$ $If$ $m<N/p,$ then $u^{p-1}\in
L^{k}(\Omega )$ for $k=Nm/(N-pm).\medskip $

\noindent (ii) $\left\vert \nabla u\right\vert ^{(p-1)}\in L^{m^{\ast
}}(\Omega ),$ where $m^{\ast }=Nm/(N-m)$. If $\overline{m}\leqq m,$ then $%
u\in W_{0}^{1,p}(\Omega ).\medskip $
\end{lemma}

\begin{remark}
The estimates on $u$ and $\left\vert \nabla u\right\vert $ are obtained in
the case $m<\overline{m}$ by using the classical test functions $\phi
_{\beta ,\varepsilon }(T_{k}(u))$, where $\phi _{\beta ,\varepsilon }(w)=$ $%
\int_{0}^{w}(\varepsilon +\left\vert t\right\vert )^{-\beta }dt$, for given
real $\beta <1.$ Let us recall the proof in the case $m\geqq \overline{m},$ $%
p<N.$ Then $L^{m}(\Omega )\subset W^{-1,p^{\prime }}(\Omega ),$ thus, from
uniqueness, $u\in W_{0}^{1,p}(\Omega )$ and $u$ is a variational solution.
If $m=\overline{m},$ then $m^{\ast }=p^{\prime },$ and the conclusion
follows. Suppose $m>\overline{m},$ equivalently $m^{\ast }>p^{\prime }.$ For
any $\sigma >p,$ for any $F\in (L^{\sigma }(\Omega ))^{N},$ there exists a
unique weak solution $w$ in $W_{0}^{1,\sigma }(\Omega )$ of the problem%
\begin{equation*}
-\Delta _{p}w= {div}(\left\vert F\right\vert ^{p-2}F)\hspace{0.5cm}\text{%
in }\Omega ,\qquad w=0\quad \text{on }\partial \Omega ,
\end{equation*}%
see \cite{Iw}, \cite{KiZh}, \cite{KiZh2}. Let $v$ be the unique solution in $%
W_{0}^{1,1}(\Omega )$ of the problem%
\begin{equation}
-\Delta v=f\hspace{0.5cm}\text{in }\Omega ,\qquad v=0\quad \text{on }%
\partial \Omega .  \label{lou}
\end{equation}%
Then from the classical Calderon-Zygmund theory, $v\in W^{2,m}(\Omega )$,
then $\left\vert \nabla v\right\vert \in L^{m^{\ast }}(\Omega ).$ Let $F$ be
defined by $\left\vert F\right\vert ^{p-2}F=\nabla v.$ Then $F\in (L^{\sigma
}(\Omega ))^{N},$ with $\sigma =(p-1)m^{\ast }>p.$ Then $-\Delta
_{p}w=-\Delta v=f,$ thus $w=u.$ Then $u\in W_{0}^{1,\sigma }(\Omega ),$ thus
$\left\vert \nabla u\right\vert ^{(p-1)}\in L^{m^{\ast }}(\Omega ).$
\end{remark}

We also obtain local estimates:

\begin{lemma}
\label{lem}Let $u\in W_{loc}^{1,p}(\Omega )$ such that
\begin{equation*}
-\Delta _{p}u=f\hspace{0.5cm}\text{in }\Omega ,
\end{equation*}%
with $f\in L_{loc}^{m}(\Omega ),$ $1<m<N,$ and $m>\overline{m}.$ Then $%
\left\vert \nabla u\right\vert ^{p-1}\in L_{loc}^{m^{\ast }}(\Omega ).$
Furthermore, for any balls $B_{1}\subset \subset B_{2}\subset \subset \Omega
,$ $\left\Vert \left\vert \nabla u\right\vert ^{p-1}\right\Vert _{L^{m^{\ast
}}\left( \overline{B_{1}}\right) }$ is bounded by a constant which depends
only on $N,p,B_{1},B_{2}$ and $\left\Vert u\right\Vert _{W^{1,p}\left(
B_{2}\right) }.$
\end{lemma}

\begin{proof}
We consider again the function $v$ defined in (\ref{lou}), and set $%
\left\vert F\right\vert ^{p-2}F=\nabla v.$ Then $F\in (L^{\sigma }(\Omega
))^{N}$ with $\sigma =(p-1)m^{\ast },$ and $u\in W_{loc}^{1,p}(\Omega )$ is
a solution of the problem
\begin{equation*}
-\Delta _{p}u= {div}(\left\vert F\right\vert ^{p-2}F)\hspace{0.5cm}\text{%
in }\Omega .
\end{equation*}%
Then, from \cite{KiZh}, $u\in W_{loc}^{1,\sigma }(\Omega )$ and for any
balls $B_{1}\subset \subset B_{2}\subset \subset \Omega ,$ $\left\Vert
u\right\Vert _{W^{1,\sigma }\left( B_{1}\right) }$ is controlled by the norm
$\left\Vert u\right\Vert _{W^{1,p}\left( B_{2}\right) }.\medskip $
\end{proof}

Next we consider problem (\ref{equ}) in the case $q<\tilde{q},$ where $%
\tilde{q}$ is defined at (\ref{defi}).

\begin{theorem}
\label{pre}Let $0<q<\tilde{q}$, $N\geqq 2.$ Let $H$ be a Caratheodory
function on $\Omega \times \mathbb{R}$ such that
\begin{equation}
\left\vert H(x,u,\xi )\right\vert \leqq g(x)+C\left\vert \xi \right\vert
^{q},  \label{lap}
\end{equation}%
where $g\in L_{loc}^{N+\varepsilon }\left( \Omega \right) ,$ $C>0.$ Let $%
u\in W_{loc}^{1,p}\left( \Omega \right) $ be any weak solution of problem (%
\ref{equ}). Then $u\in C^{1,\alpha }\left( \Omega \right) $ for some $\alpha
\in \left( 0,1\right) .$ Moreover for any balls $B_{1}\subset \subset
B_{2}\subset \subset \Omega ,$ $\left\Vert u\right\Vert _{C^{1,\alpha
}\left( \overline{B_{1}}\right) }$ is bounded by a constant which depends
only on $N,p,B_{1},B_{2},$ $\left\Vert g\right\Vert _{L^{N+\varepsilon
}\left( B_{2}\right) },$ and the norm $\left\Vert u\right\Vert
_{W^{1,p}\left( B_{2}\right) }.$
\end{theorem}

\begin{proof}
Since $u\in W_{loc}^{1,p}\left( \Omega \right) ,$ the function $%
f=-H(x,u,\nabla u)$ satisfies $f$ $\in L_{loc}^{m_{0}}\left( \Omega \right) $
from (\ref{lap}), with $m_{0}=p/q>1.$ Notice that $q<\tilde{q}$ is
equivalent to $m_{0}>$ $\overline{m}.$ If $m_{0}>N,$ then from \cite[Theorem
1.2]{DuMi1}, $\left\vert \nabla u\right\vert \in L_{loc}^{\infty }\left(
\Omega \right) $ and we get an estimate of $\left\Vert \left\vert \nabla
u\right\vert \right\Vert _{L^{\infty }\left( B_{1}\right) }$ in terms of the
norm $\left\Vert u\right\Vert _{W^{1,p}\left( B_{2}\right) }$ and $%
\left\Vert g\right\Vert _{L^{N+\varepsilon }\left( B_{2}\right) }.$ Then $%
u\in C\left( \Omega \right) ,$ $f\in L_{loc}^{\infty }\left( \Omega \right) ,
$ hence $u\in C^{1,\alpha }\left( \Omega \right) $ for some $\alpha \in
\left( 0,1\right) ,$ see \cite{To}.$\medskip $

Next suppose that $m_{0}<N.$ Then from Lemma \ref{lem} we have $\left\vert
\nabla u\right\vert ^{p-1}\in L^{(p-1)m_{0}^{\ast }}(\Omega )$. In turn,
from (\ref{lap}), $f\in L_{loc}^{m_{1}}\left( \Omega \right) $ with $%
m_{1}=(p-1)m_{0}^{\ast }/q.$ Note that $m_{1}/m_{0}=N(p-1)/(qN-p)>1$ since $%
q<\tilde{q}.$ By induction, starting from $m_{1},$ as long as $m_{n}<N,$ we
can define $m_{n+1}=(p-1)m_{n-1}^{\ast }/q,$ and we find $m_{n}<m_{n+1}.$ If
$m_{n}<N$ for any $n,$ then the sequence converges to $\lambda =N(q-p+1)/q,$
which is impossible since $p/q<\lambda $ and $q<\tilde{q}.$ Then there
exists $n_{0}$ such that $m_{n_{0}}\geqq N.$ If $n_{0}=N,$ or if $m_{0}=N$
we modify a little $m_{0}$ in order to avoid the case. Then we conclude from
above.\medskip
\end{proof}

\begin{remark}
The result, which holds without any assumption on the sign of $H$, is sharp.
Indeed for $\tilde{q}<q<p<N,$ the problem $-\Delta _{p}u=\left\vert \nabla
u\right\vert ^{q}$ in $B(0,1)$ with $u=0$ on $\partial B(0,1)$ admits the
solution
\begin{equation*}
x\longmapsto u_{C}(x)=C(\left\vert x\right\vert ^{-\frac{p-q}{q+1-p}}-1),
\end{equation*}%
for suitable $C>0,$ and $u_{C}\in W_{0}^{1,p}\left( \Omega \right) $ for $%
\tilde{q}<q.$\medskip
\end{remark}

Next we consider the absorption case, and for simplicity the model problem:

\begin{theorem}
Let $p-1<q.$ Let $u$ be a nonnegative LR solution of
\begin{equation*}
-\Delta _{p}u+\left\vert \nabla u\right\vert ^{q}=0\hspace{0.5cm}\text{in }%
\Omega .
\end{equation*}%
Then $u\in L_{loc}^{\infty }\left( \Omega \right) \cap W_{loc}^{1,p}\left(
\Omega \right) ,$ and for any balls $B_{1}\subset \subset B_{2}\subset
\subset \Omega ,\left\Vert u\right\Vert _{L^{\infty }\left( B_{1}\right) }$
and $\left\Vert u\right\Vert _{W^{1,p}\left( B_{1}\right) }$ are controlled
by the norm $\left\Vert u\right\Vert _{L^{\ell }\left( B_{2}\right) }$ for
any $\ell \in (p-1,Q_{c}).$ As a consequence if $q\leqq p,$ then $u\in
C^{1,\alpha }\left( \Omega \right) $ for some $\alpha \in \left( 0,1\right)
. $ In particular $\left\Vert \left\vert \nabla u\right\vert \right\Vert
_{L^{\infty }\left( B_{1}\right) }$ is controlled by $\left\Vert
u\right\Vert _{L^{l}\left( B_{2}\right) }.$
\end{theorem}

\begin{proof}
Since $-\Delta _{p}u\leqq 0$ in $\Omega ,$ then $u\in L_{loc}^{\infty
}\left( \Omega \right) $ from \cite{KilKuTu}, and $u$ satisfies a weak
Harnack inequality: for almost any $x_{0}$ such that $B(x_{0},3\rho )\subset
\Omega ,$ and any $\ell \in (p-1,Q_{c}),$%
\begin{equation}
\sup_{B(x_{0},\rho )}u\leq C\left( \oint_{B(x_{0},2\rho )}u^{\ell }\right) ^{%
\frac{1}{\ell }},  \label{uh}
\end{equation}%
with $C=C(N,p,\ell ).$ Then in $B(x_{0},\rho ),$  $u=T_{k}(u)$ for some $k>0,
$  thus $u\in W_{loc}^{1,p}\left( \Omega \right) .$ For any $\xi \in
\mathcal{D}\left( \Omega \right) ,$ taking $u\xi ^{p}$ as a test function,
we get
\begin{eqnarray*}
\int_{\Omega }\left\vert \nabla u\right\vert ^{p}\xi ^{p}dx+\int_{\Omega
}\left\vert \nabla u\right\vert ^{q}u\xi ^{p}dx &=&-p\int_{\Omega }\xi
^{p-1}u\left\vert \nabla u\right\vert ^{p-2}\nabla u.\nabla \xi dx \\
&\leqq &\frac{1}{2}\int_{\Omega }\left\vert \nabla u\right\vert ^{p}\xi
^{p}dx+C_{p}\int_{\Omega }u^{p}\left\vert \nabla \xi \right\vert ^{p}dx.
\end{eqnarray*}%
Then for any balls $B_{1}\subset \subset B_{2}\subset \subset \Omega ,$ we
obtain that $\left\Vert \left\vert \nabla u\right\vert \right\Vert
_{L^{p}\left( B_{1}\right) }$ is bounded by a constant which depends only on
$N,p,B$ and $\left\Vert u\right\Vert _{L^{\ell }\left( B_{2}\right) }.$ If $%
q\leqq p,$ we deduce that $u\in C^{1,\alpha }\left( \Omega \right) $ and
estimates of $\left\vert \nabla u\right\vert \in L_{loc}^{\infty }\left(
\Omega \right) $ from the classical results of \cite{To}.
\end{proof}

\begin{acknowledgement}
The first and second author were supported by Fondecyt 1110268. The second
author was also by MECESUP 0711 and CNRS UMR 7350. The third author was
partially supported by Fondecyt 1110003.
\end{acknowledgement}

\end{document}